\documentclass[reqno]{amsart}

\usepackage{enumitem}
\usepackage{hyperref}
\usepackage{bbm}
\usepackage{amsfonts,amsmath,amsxtra,amsthm,amssymb,latexsym}
\usepackage{ifpdf}
\usepackage{tikz}
\usepackage{color}

\newcommand*{\mailto}[1]{\href{mailto:#1}{\nolinkurl{#1}}}
\newcommand{\arxiv}[1]{\href{http://arxiv.org/abs/#1}{arXiv:#1}}
\newcommand{\doi}[1]{\href{http://dx.doi.org/#1}{doi: #1}}
\newcommand{\msc}[1]{\href{http://www.ams.org/msc/msc2010.html?t=&s=#1}{#1}}
\newcommand{\ack}{\section*{Acknowledgments}}

\newif\iflong
\longtrue

\newtheorem{theorem}{Theorem}[section]
\newtheorem{corollary}[theorem]{Corollary}
\newtheorem{lemma}[theorem]{Lemma}

\newtheorem{remark}[theorem]{Remark}
\newtheorem{proposition}[theorem]{Proposition}
\newtheorem{hypothesis}{Hypothesis}[section]
\theoremstyle{definition}
\newtheorem{definition}[theorem]{Definition}
\newtheorem{example}[theorem]{Example}

\newcommand{\nn}{\nonumber}
\newcommand{\be}{\begin{equation}}
\newcommand{\ee}{\end{equation}}
\newcommand{\ba}{\begin{array}}
\newcommand{\ea}{\end{array}}

\newcommand{\ol}{\overline}
\newcommand{\ti}{\tilde}
\newcommand{\id}{{\mathbbm 1}}

\numberwithin{equation}{section}

 \DeclareMathOperator{\dom}{dom}
\DeclareMathOperator{\ran}{ran} 
\DeclareMathOperator{\Span}{span}\DeclareMathOperator{\ess}{ess}

\DeclareMathOperator{\ac}{ac}
\DeclareMathOperator{\loc}{loc}
\DeclareMathOperator{\sym}{sym}
\DeclareMathOperator{\vol}{vol}

\newcommand\R{{\mathbb{R}}}
\newcommand\C{{\mathbb{C}}}

\newcommand\Z{{\mathbb{Z}}}

\newcommand\gH{{\mathfrak{H}}}

\newcommand\cH{{\mathcal{H}}}
\newcommand\cA{{\mathcal{A}}}
\newcommand\cF{{\mathcal{F}}}
\newcommand\cM{{\mathcal{M}}}

\newcommand\cL{{\mathcal{L}}}
\newcommand\cG{{\mathcal{G}}}
\newcommand\cE{{\mathcal{E}}}
\newcommand\cP{{\mathcal{P}}}
\newcommand\cV{{\mathcal{V}}}
\newcommand\cI{{\mathcal{I}}}
\newcommand\cS{{\mathcal{S}}}

\newcommand\OO{{\mathcal{O}}}

\newcommand\bH{{\mathbf{H}}}
\newcommand\rh{{\mathbf{h}}}

\newcommand\rH{{\rm{H}}}
\newcommand\E{{\rm{e}}}

\newcommand\I{{\rm{i}}}

\def\wt#1{{{\widetilde #1} }}

\begin{document}

\title[Quantum graphs on antitrees]{Quantum Graphs\\ on Radially Symmetric Antitrees}

\author[A. Kostenko]{Aleksey Kostenko}
\address{Faculty of Mathematics and Physics\\ University of Ljubljana\\ Jadranska ul.\ 21\\ 1000 Ljubljana\\ Slovenia\\ and Faculty of Mathematics\\ University of Vienna\\ 
Oskar-Morgenstern-Platz 1\\ 1090 Vienna\\ Austria}
\email{\mailto{Aleksey.Kostenko@fmf.uni-lj.si};\ \mailto{Oleksiy.Kostenko@univie.ac.at}}
\urladdr{\url{http://www.mat.univie.ac.at/~kostenko/}}

\author[N. Nicolussi]{Noema Nicolussi}
\address{Faculty of Mathematics\\ University of Vienna\\
Oskar-Morgenstern-Platz 1\\ 1090 Vienna\\ Austria}
\email{\mailto{noema.nicolussi@univie.ac.at}}


\thanks{{\it Research supported by the Austrian Science Fund (FWF) 
under Grants No.\ P 28807 (A.K. and N.N.) and W 1245 (N.N.) and by the Slovenian Research Agency (ARRS) under Grant No.\ J1-9104 (A.K.).}}
\thanks{J.\ Spectral Theory, {\it  to appear}}
\thanks{\arxiv{1901.05404}}


\begin{abstract}
We investigate spectral properties of Kirchhoff Laplacians on radially symmetric antitrees. This class of metric graphs admits a lot of symmetries, which enables us to obtain a decomposition of the corresponding Laplacian into the orthogonal sum of Sturm--Liouville operators. In contrast to the case of radially symmetric trees, the deficiency indices of the Laplacian defined on the minimal domain are at most one and they are equal to one exactly when the corresponding metric antitree has finite total volume. In this case, we provide an explicit description of all self-adjoint extensions including the Friedrichs extension. 

Furthermore, using the spectral theory of Krein strings, we perform a thorough spectral analysis of this model. In particular, we obtain discreteness and trace class criteria, a criterion for the Kirchhoff Laplacian to be uniformly positive and provide spectral gap estimates. We show that the absolutely continuous spectrum is in a certain sense a rare event, however, we also present several classes of antitrees such that the absolutely continuous spectrum of the corresponding Laplacian is  $[0,\infty)$.  
\end{abstract}

\maketitle

{\scriptsize{\tableofcontents}}

\section{Introduction}\label{sec:I}

This paper is devoted to one particular class of infinite quantum graphs, namely {\em Kirchhoff Laplacians on radially symmetric antitrees}. 
Antitrees appear in the study of {\em discrete Laplacians} on graphs at least since the 1980's (see \cite{dk} and also \cite[Section 2]{clmp}) and they attracted a considerable attention after the work of Wojciechowski \cite{woj11}. 
More precisely,  Wojciechowski used them in  \cite{woj11} (see also \cite[\S 6]{klw} and \cite{ghm}) to construct graphs of polynomial volume growth for which the (discrete) physical Laplacian is stochastically incomplete and the bottom of the essential spectrum is strictly positive, which is in sharp contrast to the manifold setting (cf. \cite{bro}, \cite{g86}, \cite{g99}). These apparent discrepancies were resolved later using the notion of intrinsic metrics, with antitrees appearing as key examples for certain thresholds (see \cite{folz, hkw13, huang,k15}). During the recent years, antitrees were also actively studied from other perspectives and we only refer to a brief selection of articles \cite{bkw15}, \cite{brke13}, \cite{clmp}, \cite{gs13}, \cite{sadel}, where further references can be found.
 
 In this paper, we consider antitrees from the perspective of quantum graphs and perform a detailed spectral analysis of the Kirchhoff Laplacian on radially symmetric antitrees. Our discussion includes characterization of self-adjointness and a complete description of self-adjoint extensions, spectral gap estimates and spectral types (discrete, singular and absolutely continuous spectrum). 
 
 Before proceeding further, let us first recall necessary definitions. 
 Let $\cG_d = (\cV, \cE)$ be a connected, simple (no loops or multiple edges) combinatorial graph. Fix a root vertex $o \in \cV$ and then order the graph with respect to the combinatorial spheres $S_n$, $n \in \Z_{\ge 0}$ (notice that $S_0=\{o\}$). 
 
 \begin{definition}\label{def:AT} 
 A connected simple rooted (infinite) graph $\cG_d$ is called an {\em antitree} if 
every vertex in $S_n$, $n\ge 1$\footnote{By definition, the root  $o$ is connected to all vertices in $S_1$ and no vertices in $S_k$, $k\ge2$.}, is connected to all  vertices in $S_{n-1}$ and $S_{n+1}$ and no vertices in $S_k$ for all $|k-n|\neq 1$.
 \end{definition}
 
 Notice that combinatorial antitrees admit radial symmetry and every antitree is uniquely determined by its sphere numbers $s_n =\#S_n$, $n\ge 0$ (see Figure \ref{fig:antitree}).  
 
 If every edge of $\cG_d$ is assigned a length $|e|\in (0,\infty)$, then $\cG=(\cG_d,|\cdot|)$ is called a {\em metric graph}. Upon identifying each edge $e$ with the interval of length $|e|$, $\cG$ may be considered as a ``network" of intervals glued together at the vertices. In the following we shall denote combinatorial and metric antitrees  by $\cA_d$ and, respectively, $\cA$. The analogue of the Laplace--Beltrami operator for metric graphs is the {\em Kirchhoff Laplacian} $\bH$ (or Kirchhoff--Neumann Laplacian, see Section \ref{ss:II.02}), also called a {\em quantum graph}. It acts as an edgewise (negative) second derivative $f_e \mapsto -\frac{d^2}{dx_e^2} f_e$, $e \in \cE$, and is defined on edgewise $H^2$-functions satisfying continuity and Kirchhoff conditions at the vertices  (we refer to \cite{bcfk06, bk13, ekkst08, EKMN, kn19, post} for more information and references).

\begin{figure}
	\begin{center}
		\begin{tikzpicture}    [%
		,scale=.6
		,every node/.style={scale=.6}]
		\draw								(0,0) -- (-1, 1.5) ;
		\draw								(0,0) -- (1, 1.5) ;
		
		\draw								(-1,1.5) -- (-2, 3) ;
		\draw								(-1,1.5) -- (0, 3) ;
		\draw								(-1,1.5) -- (2, 3) ;
		\draw								(1,1.5) -- (-2, 3) ;
		\draw								(1,1.5) -- (0, 3) ;
		\draw								(1,1.5) -- (2, 3) ;
		
		\draw								(-2, 3) -- (-3, 4.5);
		\draw								(-2, 3) -- (-1, 4.5);
		\draw								(-2, 3) -- (1, 4.5);
		\draw								(-2, 3) -- (3, 4.5);
		\draw								(0, 3) -- (-3, 4.5);
		\draw								(0, 3) -- (-1, 4.5);
		\draw								(0, 3) -- (1, 4.5);
		\draw								(0, 3) -- (3, 4.5);
		\draw								(2, 3) -- (-3, 4.5);
		\draw								(2, 3) -- (-1, 4.5);
		\draw								(2, 3) -- (1, 4.5);
		\draw								(2, 3) -- (3, 4.5);
		
		\draw [dashed, gray]			(-3, 4.5) -- (-4, 6);
		\draw [dashed , gray]			(-3, 4.5) -- (-2, 6);
		\draw [dashed, gray]			(-3, 4.5) -- (0, 6);
		\draw [dashed , gray]			(-3, 4.5) -- (2, 6);
		\draw [dashed , gray]			(-3, 4.5) -- (4, 6);
		\draw [dashed , gray]			(-1, 4.5) -- (-4, 6);
		\draw [dashed , gray]			(-1, 4.5) -- (-2, 6);
		\draw [dashed , gray]			(-1, 4.5) -- (0, 6);
		\draw [dashed , gray]			(-1, 4.5) -- (2, 6);
		\draw [dashed , gray]			(-1, 4.5) -- (4, 6);
		\draw [dashed , gray]			(1, 4.5) -- (-4, 6);
		\draw [dashed , gray]			(1, 4.5) -- (-2, 6);
		\draw [dashed , gray]			(1, 4.5) -- (0, 6);
		\draw [dashed , gray]			(1, 4.5) -- (2, 6);
		\draw [dashed , gray]			(1, 4.5) -- (4, 6);
		\draw [dashed , gray]			(3, 4.5) -- (-4, 6);
		\draw [dashed , gray]			(3, 4.5) -- (-2, 6);
		\draw [dashed , gray]			(3, 4.5) -- (0, 6);
		\draw [dashed , gray]			(3, 4.5) -- (2, 6);
		\draw [dashed , gray]			(3, 4.5) -- (4, 6);
		
		\draw	[thin, ->, dashed]			(4, 1.5) -- (1.5, 1.5);
		\draw	[thin, ->, dashed]			(4, 3) -- (2.5, 3);
		\draw	[thin, -> , dashed]			(4, 4.5) -- (3.5, 4.5);
		\draw	[thin, -> , dashed]			(4, 0) -- (0.5, 0);
		\node  at (4.5, 0) {\Large $S_0$};
		\node  at (4.5, 1.5) {\Large $S_1$};
		\node  at (4.5, 3) {\Large $S_2$};
		\node  at (4.5, 4.5) {\Large $S_3$};	
		\filldraw[color=blue]
		(0,0) circle (2pt) 
		(-1, 1.5) circle (2pt)
		(1, 1.5) circle (2pt)
		(-2, 3) circle (2pt)
		(0, 3) circle (2pt)
		(2, 3) circle (2pt)
		(-3, 4.5) circle (2pt)
		(-1, 4.5) circle (2pt)
		(1, 4.5) circle (2pt)
		(3, 4.5) circle (2pt);
		\end{tikzpicture}
	\end{center}
	\caption{Antitree with sphere numbers $s_n = n+1$.}\label{fig:antitree}
\end{figure}
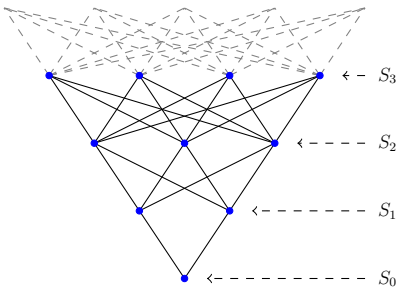

 Our approach employs the high degree of symmetry and this naturally demands symmetry assumptions also on the choice of edge lengths \iflong(Remark \ref{rem:infdi}(i) shows this is indeed necessary for our results)\fi: 
 
\begin{hypothesis}\label{hyp:symm}We shall assume that the metric antitree $\cA$ is {\em radially symmetric}, that is, for each $n\ge 0$, all edges connecting combinatorial spheres $S_n$ and $S_{n+1}$ have the same length, say $\ell_n >0$. 
 \end{hypothesis}
 
One of our main motivations is Lemma 8.9 in \cite{kn19}. More precisely, the symmetry of antitrees structure turned out useful in studying isoperimetric estimates and we were even able to compute explicitly the bottom of the essential spectrum of some (non-equilateral) quantum graphs (see \cite[\S 8.2]{kn19}). Despite an enormous interest in quantum graphs during the last two decades, to the best of our knowledge a detailed discussion of their spectral properties without further restrictions on edges lengths (for instance, one of the most common assumptions is $\inf_{e\in\cE}|e|>0$) has so far been obtained only for a few models and the most studied ones are {\em radially symmetric trees} (see e.g. \cite{bf,car00, ess, ns00, ns01, sol04}). However, the assumption that $\cG$ is a tree prevents many interesting phenomena to happen (for instance, by \cite[Lemma 8.1]{kn19}, in this case the Kirchhoff Laplacian, actually, its Friedrichs extension, is boundedly invertible if and only if $\sup_{e \in \cE}|e|< \infty$; in fact, this condition is only necessary in general \cite{sol03}). Hence our goal in this work is to present a model which can be thoroughly analyzed but still exhibits in some sense rich spectral behavior. 

Let us now briefly describe the content of the paper and our main results. To some extent we follow the approach developed by Naimark and Solomyak for radially symmetric trees (see \cite{ns00, ns01} and also \cite{car00, sol03, sol04}) and use some ideas from \cite{brke13}, where discrete Laplacians on radially symmetric ``weighted" graphs have been analyzed. However, some modifications are necessary since comparing to \cite{car00, ns01, sol04} we are dealing with a completely different class of graphs (antitrees have a lot of cycles) and, in contrast to discrete Laplacians \cite{brke13}, we have to deal with unbounded operators (even when restricting to compact subsets of a metric graph) and in this case a search for {\em reducing subspaces} is a rather complicated task 
\footnote{After the submission of our paper we learned about the preprint \cite{bl19} dealing with a similar decomposition in the general case of family preserving metric graphs, which includes antitrees as a particular example. However, the main focus of \cite{bl19} is on the existence of a decomposition in a rather general situation, whereas in our work we use it mainly as a starting point for the spectral analysis.}

 First of all, the radial symmetry of $\cA$ naturally hints to consider the space $\cF_{\sym}$ of radially symmetric functions (w.r.t. the root $o \in \cV$). It turns out that $\cF_{\sym}$ is indeed reducing for the pre-minimal Kirchhoff Laplacian $\bH_0$ (this means that $\bH_0$ as well as its closure $\bH = \overline{\bH}_0$, the minimal Kirchhoff Laplacian, commutes with the orthogonal projection onto $\cF_{\sym}$) and its restriction $\bH_0\upharpoonright {\cF_{\sym}}$ is unitarily equivalent to a pre-minimal Sturm--Liouville operator $\rH_0$ defined in $L^2((0,\cL);\mu)$ by the differential expression
\be
\tau := -\frac{1}{\mu(t)}\frac{d}{dt}\mu(t)\frac{d}{dt}, \qquad  \mu(t) = \sum_{n\ge 0}s_ns_{n+1}\id_{[t_{n},t_{n+1})}(t),
\ee
and subject to the Neumann boundary condition at $x=0$. Here $t_0 = 0$, $t_n = \sum_{k\le n-1}\ell_k$ for all $n\ge 1$ and $\cL = \sum_{n\ge 0}\ell_n$ (see Section \ref{ss:3.1}). 
Moreover, the remaining part of $\bH = \overline{\bH}_0$ decomposes into an infinite sum of self-adjoint (regular) Sturm--Liouville operators (see Theorem \ref{th:decomp}; its proof is given in Sections \ref{sec:II} and \ref{sec:III}). This decomposition is the starting point of our analysis since it enables us to investigate $\bH$ using the well-developed spectral theory of Sturm--Liouville operators. For example, this immediately provides a self-adjointness criterion together with a complete description of self-adjoint extensions of $\bH$ (see Section \ref{sec:sa}). Namely, since all the summands in \eqref{eq:Hdecomp} except $\rH = \overline{\rH}_0$ are self-adjoint operators, we reduce the problem to the study of the operator $\rH_0$. Employing Weyl's limit point/limit circle classification, we obtain in Theorem \ref{th:SA} that deficiency indices of $\bH$ are at most $1$. Moreover, $\bH$ is self-adjoint if and only if $\cA$ has {\em infinite total volume}, i.e.
\[
	\vol(\cA) := \sum_{e \in \cE} |e| = \sum_{n \ge 0} s_n s_{n+1} \ell_n = \int_0^\cL \mu(t)dt = \infty. 
\]
If $\cA$ has finite total volume, $\vol(\cA) < \infty$, all self-adjoint extensions can be described through a single boundary condition (in particular, this also provides a description of the domain of the Friedrichs extension). Moreover, all of their spectra are purely discrete and eigenvalues satisfy Weyl's law (see Corollary \ref{cor:finvol}). 

If $\vol(\cA) =  \infty$, i.e., $\bH$ is self-adjoint, it was already observed in \cite[Section 8.2]{kn19} that $\sigma(\bH)$ is not necessarily discrete. In Section \ref{sec:discr}, we characterize the cases when $\bH$ has purely discrete spectrum and when its resolvent $\bH^{-1}$ belongs to the trace class  (see Theorem \ref{th:discr} and Theorem \ref{th:trace}). Let us stress that our main tool is the spectral theory of Krein strings \cite{kakr74} (see also \cite{dymc76}). More precisely, by a simple change of variables $\rH$ can be transformed into the string form (see \eqref{eq:wtH}) and then one simply needs to use the corresponding results from \cite{kakr58,kakr74}. 
Section \ref{sec:specest} is devoted to {\em spectral estimates}, i.e., the investigation of the bottom of the spectrum $\lambda_0(\bH)$ of $\bH$, $\lambda_0(\bH) := \inf \sigma(\bH)$. This can be solved again by using the results of Kac and Krein from \cite{kakr58}. More precisely, we characterize the positivity of $\lambda_0(\bH)$ (Theorem \ref{th:SpEst} and Theorem \ref{th:SpEstEss}) and derive two-sided estimates (Remark \ref{rem:SpEst01}). Let us also mention at this point that the decomposition \eqref{eq:Hdecomp} indicates the way to compute the isoperimetric constant of a radially symmetric antitree (see Theorem \ref{th:alpha}) and hence it is interesting to compare Theorem \ref{th:SpEst} and Theorem \ref{th:SpEstEss}  with the estimates obtained recently in \cite{kn19}  (see Remark \ref{rem:cheeger}).

To our best knowledge, the theory of Krein strings is applied in the context of quantum graphs for the first time. In fact, most of the analysis in Sections \ref{sec:discr} and \ref{sec:specest} can be performed with the help of Muckenhoupt inequalities \cite{muc} since the questions addressed in these sections allow a variational reformulation (in particular, Solomyak used this approach in \cite{sol04} to investigate quantum graphs on radially symmetric trees). However, spectral theory of strings enables us to treat more subtle problems (like the study of the structure of the essential spectrum of $\bH$). In Section \ref{sec:ACspec}, we employ the recent results from \cite{bd17} and \cite{ek18} on the absolutely continuous spectrum of strings to construct several classes of antitrees with absolutely continuous spectrum supported on $[0,\infty)$.  For instance, if 
\begin{align}
	\inf_{n\ge 0} \ell_n & >0, & \sum_{n=1}^\infty \Big ( \frac{s_{n+2}}{s_n} -1 \Big )^2 & < \infty,
\end{align}
 then $\sigma_{\ac} (\bH) = [0,\infty)$ (see Theorem \ref{cor:ac01}). Notice that to prove this claim we employ the analog of the Szeg\H{o} theorem for strings recently established by Bessonov and Denisov \cite{bd17}. Antitrees with polynomially growing sphere numbers satisfy the last assumption, however, it can be shown that in this case the usual trace class arguments do not apply (see Remark \ref{rem:8.4}). Let us also emphasize that similar to the case of trees quantum graphs typically have purely singular spectrum in the case of antitrees (see Section \ref{sec:SCspec}). However, to the best of our knowledge, the only known examples of quantum graphs on trees having nonempty absolutely continuous spectrum are {\em eventually periodic radially symmetric trees} (see \cite[Theorem 5.1]{ess}).  
 
In the final section we demonstrate our results by considering two special classes of antitrees and complement the results of \cite[Section 8.2]{kn19}. In Section \ref{ss:expAT} we consider antitrees with exponentially increasing sphere numbers and demonstrate that in this case there are a lot of similarities with the spectral properties of quantum graphs on radially symmetric trees. Antitrees with polynomially increasing sphere numbers are treated in Section \ref{ss:polAT} and this class of quantum graphs exhibits a number of interesting phenomena. For example, one can show a transition from absolutely continuous spectrum supported on $[0,\infty)$ to purely discrete spectrum (see Corollary \ref{cor:Hqs}).

\iflong{}In Appendix \ref{app:atan2} we discuss the eigenvalues of a special class of regular Sturm--Liouville operators and in Appendix \ref{sec:ATbndgeom} we collect several examples of antitrees whose degree function takes finitely many values and the absolutely continuous spectrum of the corresponding Laplacian is  $[0,\infty)$.

Finally, let us stress that our approach based on spectral theory of Krein strings enables us (without almost no effort) to extend most of the results obtained in this paper to arbitrary second order differential operators (namely, one can replace the second derivative by a weighted Sturm--Liouville operator $ -\frac{1}{r(x)}\frac{d}{dx}p(x)\frac{d}{dx} $ or even by a string differential expression $-\frac{d^2}{d\omega(x)dx}$), of course keeping the radial symmetry assumption on the coefficients.\fi

\section{Decomposition of $L^2(\cA)$}\label{sec:II}

\subsection{Auxiliary subspaces}\label{ss:II.1}
Let $\cA$ be a metric radially symmetric antitree with sphere numbers $\{s_n\}_{n\ge 0}$ and lengths $\{\ell_n\}_{n\ge 0}$. Upon identifying every edge $e$ with a copy of the interval $\cI_e = [0,|e|]$  and considering $\cA$ as the union of all edges glued together at certain endpoints, one can introduce the Hilbert space $L^2(\cA)$ of functions $f\colon \cA\to \C$ as $L^2(\cA) = \oplus_{e} L^2(e)$. Next, denote
  \begin{align*}
  &t_n := \sum_{j=0}^{n-1} \ell_j, &I_n := [t_n, t_{n+1}),
  \end{align*}
  and let $\cH_n := \C^{s_n s_{n+1}}$, $n\ge 0$. Notice that $s_n s_{n+1}$ is the number of edges in $\cE_n^+$, where $\cE_n^+$ is the set of edges connecting $S_n$ with $S_{n+1}$. Enumerating the vertices in each sphere, let each entry $a_{ij}$ of some ${\bf a} = (a_{ij})_{i,j} \in \cH_n$ correspond to a coefficient of the edge $e \in \cE^+_n$ connecting the $i$-th vertex of $S_n$ with the $j$-th vertex of $S_{n+1}$. 
 Moreover, we can identify each function $f \colon \cA\to \C$ in a natural way with the sequence of functions ${\bf f} = ({\bf f}^n)_{n \geq 0}$ such that ${\bf f}^n \colon I_n \to \cH_n$. In fact, ${\bf f}^n$ is given by
  \begin{align}\label{eq:bff_n}
	{\bf f}^n_{i,j} (t) := f(x_{ij}(t)), \quad t \in I_n,  
  \end{align}
 where $x_{ij}(t)$ is the unique $x \in \cA$, such that $|x| = t$ and $x$ lies on the edge connecting the $i$-th vertex in $S_n$ with the $j$-th vertex of $S_{n+1}$. Notice that the map 
 \be\label{eq:U}
 \begin{array}{cccc}
 U\colon & L^2(\cA) &\to & \oplus_{n\ge 0}L^2(I_n;\cH_n) \\[2mm]
  & f & \mapsto & {\bf f} = ({\bf f}^n)_{n \geq 0}
 \end{array}
 \ee
is an isometric isomorphism since 
 \begin{equation} \label{eq:prodvec}
	 \big(f, g\big)_{L^2(\cA)} = \sum_{n \geq 0} \int_{I^n} ({\bf f}^n (t),{\bf g}^n (t) )_{\cH_n} \; dt
 \end{equation}
 for all $f, g \in L^2(\cA)$. 
 Next we introduce the following subspaces:
 \begin{align*}
	 &\cH^{\sym}_n := \big\{ {\bf a}  \in \cH_n | \; a_{ij} = a_{11} \; \forall i,j \big\},  \\[2mm]
	 &\cH^+_n := \Big\{ {\bf a} \in \cH_n | \; a_{ij} = a_{i1}  \; \forall i,j, \text{ and }  \sum_{i,j} a_{ij} = \sum_{i} a_{i1} =0 \Big\}, \\
	 &\cH^-_n := \Big\{ {\bf a} \in \cH_n | \; a_{ij} = a_{1j}  \; \forall i,j,  \text{ and }   \sum_{i,j} a_{ij} = \sum_{j} a_{1j} = 0\Big\}, \\
	 &\cH_n^0 :=\Big \{ {\bf a}  \in \cH_n | \; \sum_{j} a_{ij} = 0 \; \forall i \text{ and } \sum_{i} a_{ij} = 0 \; \forall j  \Big\}. 
 \end{align*}
 It is straightforward to check that the above spaces are mutually orthogonal and their dimensions are given by
 \begin{align*}
	 &\dim(\cH^{\sym}_n) = 1, &&\dim(\cH_n^0) =  (s_{n}-1) (s_{n+1} -1), \\[1mm]
	 &\dim(\cH^+_n) = s_n -1, &&\dim(\cH^-_n) = s_{n+1} -1.
 \end{align*}
 Hence $\cH_n$ admits the decomposition
 \begin{equation} \label{eq:decompfinite}
	 \cH_n = \begin{cases} \cH^{\sym}_n \oplus \cH^-_n, & n=0\\
	 \cH^{\sym}_n \oplus \cH^+_n \oplus \cH^-_n \oplus \cH_n^0, & n\ge 1 \end{cases}.
 \end{equation}
Notice also that if $s_n = 1$ for some $n\ge 1$, then $\cH^+_n = \cH_n^0 = \cH_{n-1}^0 = \cH^-_{n-1} = \{0\}$.

One can also describe the above subspaces by identifying $\cH_n$ with the tensor product $\C^{s_n}\otimes \C^{s_{n+1}}$. For example, setting 
\begin{align}\label{eq:1n}
{\bf 1}_{s_n} & := \big(\underbrace{1,1,\dots,1}_{s_n}\big)\in \C^{s_n}, & {\bf 1}^n & :={\bf 1}_{s_n}\otimes {\bf 1}_{s_{n+1}}\in \cH_n,
\end{align}
for all $n\ge 0$, 
 we get
 \be\label{eq:Hsym1n}
 \cH^{\sym}_n = \Span\{{\bf 1}^n\}.
 \ee
Moreover, denote
\[
\omega_n := \E^{2\pi \I/s_n},\quad n\ge 0,
\]
and set
\be\label{eq:a_sn}
{\bf a}^j_{s_n} := \{\omega_n^j,\dots,\omega_n^{j(s_n-1)},1\}\in \C^{s_n},\quad j\in \{1,\dots,s_n\}.
\ee
Notice that $\{{\bf a}^j_{s_n}\}_{j=1}^{s_n}$ forms an orthogonal basis in $\C^{s_n}$ for all $n\ge 0$. In particular, ${\bf a}^{s_n}_{s_n} = {\bf 1}_{s_n}$ and $\|{\bf a}^j_{s_n}\|^2 = s_n$. Hence setting
\be\label{eq:a_nij}
{\bf a}_n^{i,j}:= {\bf a}^i_{s_{n}}\otimes {\bf a}^j_{s_{n+1}} \in \cH_n,
\ee
where $1\le i\le s_n$ and $1\le j\le s_{n+1}$, we easily get
\be\label{eq:nHn}
\begin{split}
 \cH^+_n &= \Span\big\{{\bf a}^i_{s_{n}}\otimes {\bf 1}_{s_{n+1}} |\, 1\le i < s_n\big\} = \Span\big\{{\bf a}_n^{i,s_{n+1}} |\, 1\le i < s_n\big\},\\
 \cH^-_n &= \Span\big\{{\bf 1}_{s_{n}}\otimes {\bf a}^j_{s_{n+1}} |\, 1 \le j < s_{n+1}\big\} =\Span\big\{{\bf a}_n^{s_n,j} |\, 1\le j < s_{n+1}\big\},\\
\cH_n^0 &= \Span\big\{{\bf a}_n^{i,j}|\, 1\le i< s_n,\ 1\le j < s_{n+1}\big\}.
\end{split}
\ee
Finally, observe that 
\be\label{eq:norm_a}
\|{\bf a}_n^{i,j}\|^2 = s_ns_{n+1}
\ee
for all $1\le i\le s_n$, $1\le j\le s_{n+1}$ and $n\ge 0$.

\subsection{Definition of the subspaces}\label{ss:II.2}

The decomposition \eqref{eq:decompfinite} naturally induces a decomposition of the Hilbert space $L^2(\cA)$. 
First consider the subspace 
\be\label{eq:Fsym2}
\cF_{\sym} := \{f \in L^2(\cA)| \; {\bf f}^n\colon I_n\to \cH^{\sym}_n ,\  n \geq 0 \}.
\ee
Clearly, it consists of functions which depend only on the distance to the root:
\be\label{eq:Fsym}
\cF_{\sym} = \{f\in L^2(\cA)|\, f(x) = f(y)\ \text{if}\ |x|=|y|\}.
\ee
Moreover, its orthogonal complement is given by
\begin{align}
\cF_{\sym}^\perp 
&=\{f \in L^2(\cA)| \; {\bf f}^n\colon I_n\to (\cH^{\sym}_n)^\perp ,\  n \geq 0 \}\label{eq:FsymPerp}\\[2mm]
&=\Big\{f \in L^2(\cA)| \; \sum_{e\in \cE_n^+} f_e \equiv 0,\  n \geq 0 \Big\}.\nn 
\end{align}

Next we need to decompose $\cF_{\sym}^\perp$. 
 Set
\be\label{eq:cFn0}
\cF_n^0: =  \{f \in L^2(\cA)| \; {\bf f}^{n}\colon I_n\to \cH_{n}^0;\ {\bf f}^{k} \equiv 0, \; k \neq n\}
\ee
for all $n\ge 1$. Taking into account the definition of $\cH_{n}^0$, it is not difficult to see that 
\begin{align*}
	\cF_n^0 = \Big\{f \in L^2(\cA)|\, f\equiv 0\ \text{on}\ \cA\setminus\cE_n^+;\ \sum_{e \in \cE^+_v}f_e =  \sum_{e \in \cE^-_u}f_e \equiv 0 \; \forall v \in S_n, u \in S_{n+1} \Big\}. 
\end{align*}
Here, for every $v \in \cV$, $\cE^+_v$ and $\cE^-_v$ denote the edges connecting $v$ with the next and, respectively, previous combinatorial spheres.

We need to be more careful with the remaining part since our aim is to find reducing subspaces for the quantum graph operator $\bH$. For every $v\in \cV\setminus o$, define the subspace $\wt \cF_v$ consisting of functions which vanish away of $\cE_v$, where $\cE_v$ is the set of edges emanating from $v$. Moreover, on the corresponding star $\cE_v$ they depend only on the distance to the root, that is,
\be\label{eq:Fv}
\wt \cF_v := \big\{f\in L^2(\cA)|\, f\equiv 0\ \text{on}\ \cA\setminus\cE_v;\ f(x)=f(y) \ \text{for a.e.}\ x,y\in\cE_v, |x|=|y|\big\}.
\ee 
Notice that $\wt \cF_v$ and $ \wt \cF_u$ are orthogonal for $u\neq v$ if $u$ and $v$ are not adjacent vertices. 
 Next for all $n\ge 1$ consider the spaces
\be\label{eq:wtFn}
	\wt{\cF}_n :=\bigoplus_{v\in S_n}\wt\cF_v,\qquad n\ge 1,
\ee
and
\be\label{eq:wtFn'}
	\cF_n := \wt{\cF}_n\ominus \cF_{\sym} =  
	\Big\{f \in \wt{\cF}_n| \; \sum_{e \in \cE_m^+} f_e \equiv 0,\  m \geq 0 \Big\}.
\ee
Notice that with respect to the decomposition \eqref{eq:decompfinite}, we have
 \begin{align}
	\cF_n =  \{f \in L^2(\cA)| \; {\bf f}^{n-1}\colon I_{n-1} \to \cH_{n-1}^-, \; {\bf f}^{n}\colon I_n\to \cH_n^+;  {\bf f}^{m}\equiv 0, \;  m \neq n-1, n \}.
 \end{align}
 
Thus, we arrive at the following result.

\begin{lemma}\label{lem:dec1}
	The Hilbert space $L^2(\cA)$ admits the decomposition
 \begin{equation} \label{eq:basicedecomp}
 L^2(\cA) = \cF_{\sym} \oplus \bigoplus_{n \geq 1}\cF_n \oplus \bigoplus_{n \geq 1} \cF_n^0. 
 \end{equation}
 \end{lemma}
 
\begin{proof}
	The orthogonality of the subspaces in \eqref{eq:basicedecomp} follows directly from \eqref{eq:prodvec} and \eqref{eq:decompfinite}. Hence we only need to show that every $f \in L^2(\cA)$ is contained in the right hand side of \eqref{eq:basicedecomp}. Since $L^2(\cA) = \oplus_{e\in\cE}L^2(e)$, it suffices to prove this claim in the case when $f$ is zero except on a single edge $e \in \cE$. Suppose that $e \in \cE^+_n$ for some $n\ge 0$. Then
for almost every $t \in I_n$ we have
	\[
		{\bf f}^n (t) = \cP_n^{\sym}( {\bf f}^n (t)) + \cP_n^+ ( {\bf f}^n (t)) + \cP_n^- ( {\bf f}^n (t))+ \cP_n^0( {\bf f}^n (t)) \in \cH_n, 
	\]
	where $\cP_n^j\colon \cH_n \to \cH_n^j$ is the orthogonal projection in $\cH_n$ onto $\cH_n^j$, $j\in \{\sym,+,-,0\}$. Define $f_j \colon \cA \to \C$ as the function identified with the sequence of functions $ {\bf f}_j = ({\bf f}_j^k)_{k \geq 0}$ given by
	\begin{align*}
		{\bf f}_j^k (t) := P_k^j ( {\bf f}^k (t)), \qquad j\in \{\sym,+,-,0\}, 
	\end{align*}
	for a.e. $t \in I_k$. 
	Then $f_j \in L^2(\cA)$ for all $j\in \{\sym,+,-,0\}$ and 
	\[
	f = f_{\sym} + f_+ + f_- + f_0.
	\]  
Since ${\bf f}_j^k (t) \in \cH^j_k$ for a.e. $t\in I_k$, we conclude that $f_{\sym} \in \cF_{\sym}$, $f_0 \in \cF_n^0$, $f_+ \in \cF_n$ and $f_- \in \cF_{n+1}$.
 \end{proof}

Our next aim is to write down explicit formulas for projections onto the subspaces in the decomposition \eqref{eq:basicedecomp}. First, for any $\tilde f  \in L^2(I_n)$ and ${\bf a} \in \cH_n$, we set $\tilde{\bf f} :=\tilde f \otimes {\bf a}$.
 Recalling that every function $f\colon \cA\to \C$ can be identified via \eqref{eq:U} with the sequence of vector-valued functions ${\bf f}=({\bf f}^n)_{n\ge 0}$, we denote
\be
	\cF^n_{{\bf a}} := \{f\in L^2(\cA)|\, {\bf f}^n = f^n\otimes {\bf a},\ f^n\in L^2(I_n);\, {\bf f}^k\equiv 0,\ k\neq n\}.
\ee
Note that the orthogonal projection $P_{\bf a}^n$ of $L^2(\cA)$ onto $\cF_{\bf a}^n$ is given by 
\begin{align} \label{eq:P_an}
	(U(P_{\bf a}^n f))(t) := \begin{cases} 0, & t\notin I_n \\
	\frac{1}{\|{\bf a}\|^2}({\bf f}^n(t), {\bf a})_{\cH_n}{\bf a}, & t\in I_n
	\end{cases},
\end{align}
where $U$ is the isometric isomorphism \eqref{eq:U}.

Combining the form of $P_{\bf a}^n$ with the decomposition \eqref{eq:decompfinite} and \eqref{eq:Hsym1n}, \eqref{eq:nHn}, we easily obtain the following result. 

\begin{lemma}\label{lem:Psym}
Let ${\bf 1}^n \in \cH_n$ and ${\bf a}_n^{i,j}\in \cH_n$, $n\ge 0$ be given by \eqref{eq:1n} and \eqref{eq:a_nij}. Then the orthogonal projections in the decomposition \eqref{eq:basicedecomp} are given by 
\begin{align}
		& P_{\sym} =  \sum_{n \geq 0} P_{{\bf 1}^n}^n, \label{eq:Psym1}\\
		& P_n^0 = \sum_{\substack{1\le i < s_n\\ 1\le j<s_{n+1}}} P_{{\bf a}_n^{i,j}}^n, \qquad n\ge 1,\label{eq:Pn0} \\
		& P_n = \sum_{j=1}^{s_{n}-1} P_{{\bf a}^{s_{n-1},j}_{n-1}}^{n-1} + \sum_{i=1}^{s_{n}-1} P_{{\bf a}_n^{i,s_{n+1}}}^n,\qquad n\ge 1. \label{eq:Pn} 
\end{align}	
\end{lemma}

\section{Reduction of the quantum graph operator}\label{sec:III}

In this section, we show that each of the spaces in the above decomposition \eqref{eq:basicedecomp} is reducing for the quantum graph operator with Kirchhoff conditions and also obtain a description of the corresponding restrictions. 

\subsection{Kirchhoff's Laplacian}\label{ss:II.02}

Let us briefly recall the definition of the Laplacian on a metric graph (for details we refer to \cite{bk13,EKMN,kn19}). Let $L^2(\cA)$ be the corresponding Hilbert space and the subspace of compactly supported $L^2$-functions will be denoted by $L^2_c(\cA)$. Moreover, denote by $H^2(\cA\setminus\cV)$ the subspace of $L^2(\cA)$ consisting of  edgewise $H^2$-functions, that is,  $f\in H^2(\cA\setminus\cV)$ if $f\in H^2(e)$ for every $e\in \cE$, where  $H^2(e)$ is the usual Sobolev space. The Kirchhoff (or Kirchhoff--Neumann) boundary conditions at every vertex $v\in\cV$ are then given by
\be\label{eq:kirchhoff}
\begin{cases} f\ \text{is continuous at}\ v \\[1mm] \sum_{e\in \cE_v}f_e'(v) =0 \end{cases},
\ee 
where 
\begin{align}\label{eq:tr_fe}
f_e(v) & :=  \lim_{x_e\to v} f(x_e), & f_e'(v) & := \lim_{x_e\to v} \frac{f(x_e) - f_e(v)}{|x_e - v|},
\end{align}
are well defined for all $f\in H^2(\cA\setminus\cV)$ and every vertex $v\in \cV$. 
Imposing these boundary conditions  and restricting to compactly supported functions we get the pre-minimal operator
$\bH_0$ acting edgewise as the (negative) second derivative $f_e \mapsto -\frac{d^2}{dx_e^2} f_e$, $e \in \cE$ on the domain
\be\label{eq:H0}
	 \dom(\bH_{0}) = \{f\in H^2(\cA\setminus\cV) \cap L^2_{c}(\cG)|\, f\ \text{satisfies}\ \eqref{eq:kirchhoff},\ v\in\cV\}.
\ee
The operator $\bH_0$ is symmetric and  its closure $\bH=\overline{\bH}_0$ is called {\em the minimal Kirchhoff Laplacian}. 

First, we need the following simple but useful fact. 

\begin{lemma}\label{lem:domH}
Let $f\in L^2(\cA)$ and ${\bf f}= Uf$ be given by \eqref{eq:U}. Then $f\in\dom(\bH_0)$ if and only if ${\bf f} = ({\bf f}^n)_{n\ge 0}$ satisfies the following conditions:
\begin{enumerate}[label=(\roman*), ref=(\roman*), leftmargin=*, widest=iiii]
\item[{\rm (i)}] ${\bf f}^n\equiv 0$ for all sufficiently large $n$,
\item[{\rm (ii)}] ${\bf f}_{i,j}^n \in H^2(I_n)$ for all $n\ge 0$,
\item[{\rm (iii)}] for all $j\in \{1,\dots,s_1\}$
\begin{align*}
{\bf f}^0_{1,j}(0+) & = {\bf f}^0_{1,1}(0+), & \sum_{j=1}^{s_1} ({\bf f}^0_{1,j})'(0+) & =0,
\end{align*}
\item[{\rm (iv)}] for all $n\ge 1$,
\[ 
\begin{array}{c}
{\bf f}^n_{i,j}(t_n+) = {\bf f}^{n-1}_{k,i}(t_{n}-) \\[2mm]
\sum_{j=1}^{s_{n+1}}({\bf f}^n_{i,j})'(t_n+)  = \sum_{k=1}^{s_{n-1}}({\bf f}^{n-1}_{k,i})'(t_n-)
\end{array},\qquad i\in \{1,\dots,s_n\}.
\]
\end{enumerate}
\end{lemma}

\begin{proof}
The proof is straightforward. We only need to mention that (i) is equivalent to the fact that $f$ is compactly supported; (ii) means that $f$ belongs to the Sobolev space $H^2$ on each edge $e\in\cE$; (iii) and (iv) are continuity and Kirchhoff's conditions at the vertices.
\end{proof}

\subsection{The subspace $\cF_{\sym}$}\label{ss:3.1}

Set $\cI_\cL = [0,\cL)$, and define the length $\cL$ and the weight function $\mu\colon \cI_\cL \to \R_{\ge 0}$ by
\begin{align}\label{eq:muL}
		 \mu(t) & = \sum_{n \geq 0} s_n s_{n+1} \id_{I_n}(t),\quad t\in [0,\cL); & \cL & = \sum_{n\ge 0}\ell_n .
\end{align}Consider the (pre-minimal) Sturm--Liouville operator $\rH_0$ defined in $L^2(\cI_\cL; \mu)$ by the differential expression
\be\label{eq:tauA}
	\tau = -\frac{1}{\mu(t)}\frac{d}{d t}\mu(t)\frac{d}{d t},
\ee
on the domain
\be
\dom(\rH_0):= \big\{f\in L^2_c(\cI_\cL; \mu)|\, f,\, \mu f' \in AC_{\loc}(\cI_\cL),\ \tau f\in L^2(\cI_\cL; \mu);\ f'(0)=0\big\}.
\ee
More concretely, $\rH_0$ acts as a negative second derivative and its domain $\dom(\rH_0)$ consists of functions $f \in L^2(\cI_\cL; \mu)$ having  compact support in $\cI_\cL$, belonging to $H^2$ on every interval $I_n$ and at each point $t_n$ satisfying the boundary conditions
\be \label{eq:bcsym}
\begin{cases} f\ \text{is continuous at}\ t_n,
	\\[1mm] s_{n-1} f'(t_n-) = s_{n+1} f'(t_n+).
\end{cases}
\ee
Here we set $s_{-1} := 0$ in the case $n=0$ for notational simplicity and the corresponding condition \eqref{eq:bcsym} reads as the Neumann boundary condition at $t=0$.

\begin{lemma}\label{lem:FsymH}
	The subspace $\cF_{\sym}$ reduces the operator $\bH_0$. Moreover, its restriction $\bH_0\upharpoonright{\cF_{\sym}}$ onto $\cF_{\sym}$ is unitarily equivalent to the operator $\rH_0$.
\end{lemma}

\begin{proof}
First let us show that  $f_{\sym}:=P_{\sym}f \in \dom(\bH_0)$ for every $f\in \dom(\bH_0)$. In fact, we need to show that ${\bf f}_{\sym} = Uf_{\sym}$ satisfies conditions (i)--(iv) of Lemma \ref{lem:domH}. 
Clearly, by continuity of $f$ and \eqref{eq:P_an}, \eqref{eq:Psym1}, ${\bf f}_{\sym}$ satisfies (i) and (ii). Moreover, both $({\bf f}_{\sym})^n_{i,j}(t_n+)$ and  $({\bf f}_{\sym})^n_{k,m}(t_{n+1}-)$ depend only on $n\ge 0$. Since ${\bf f}$ satisfies both (iii) and (iv), we obtain that $({\bf f}_{\sym})^0_{1,j}(0+)$ does not depend on $j$ and 
\begin{align*}
({\bf f}_{\sym})^n_{i,j}(t_n+) & = \frac{1}{s_{n}s_{n+1}}\big({\bf f}^n(t_n+),{\bf 1}^n\big)_{\cH_n} \\
& = \frac{1}{s_{n-1}s_{n}}\big({\bf f}^{n-1}(t_{n}-),{\bf 1}^{n-1}\big)_{\cH_{n-1}} = ({\bf f}_{\sym})^{n-1}_{k,i}(t_{n}-)
\end{align*}
for all $i\in \{1,\dots,s_n\}$ and $n\ge 1$. Similarly,
\begin{align}
\sum_{j=1}^{s_{n+1}}({\bf f}_{\sym}')^n_{i,j}(t_n+) & = \frac{1}{s_{n}}\big(({\bf f}^n)'(t_n+),{\bf 1}^n\big)_{\cH_n} = \frac{1}{s_{n}}\sum_{i,j}({\bf f}^n_{i,j})'(t_n+) \nn\\
&=\frac{1}{s_{n}}\sum_{i=1}^{s_n} \sum_{j=1}^{s_{n+1}}({\bf f}^n_{i,j})'(t_n+) =\frac{1}{s_{n}}\sum_{i=1}^{s_n} \sum_{k=1}^{s_{n-1}}({\bf f}^{n-1}_{k,i})'(t_n-)\nn\\
&= \frac{1}{s_{n}}\big(({\bf f}^{n-1})'(t_{n}-),{\bf 1}^{n-1}\big)_{\cH_{n-1}} = \sum_{k=1}^{s_{n-1}}({\bf f}_{\sym}')^{n-1}_{k,i}(t_{n}-), \label{eq:3.05}
\end{align}
which holds for all $i\in\{1,\dots,s_n\}$, $n\ge 1$. Moreover, for $n=0$ we have
\[
({\bf f}_{\sym}')^0_{1,j}(0+) = \frac{1}{s_1}\sum_{m=1}^{s_1} ({\bf f}^0_{1,m})'(0+) = 0
\]
for all $j\in \{1,\dots,s_1\}$. 
Hence $f_{\sym} = P_{\sym}f \in \dom(\bH_0)$ for all $f \in \dom(\bH_0)$. Noting that $\bH_0$ is symmetric and $\cF_{\sym}$ is clearly invariant for $\bH_0$ we conclude that $\cF_{\sym}$ is reducing for $\bH_0$.

To prove the last claim, observe that the subspace $\cF_{\sym}$ is isometrically isomorphic to the Hilbert space $L^2(\cI_\cL;\mu)$. Indeed, for every $f\in \cF_{\sym}$, set 
\be\label{eq:tilde_f}
\tilde{f}(t):= \frac{1}{s_{n}s_{n+1}}\sum_{e\in \cE_n^+} f(x_{e}(t)) = \frac{1}{\|{\bf 1}^n\|^2}({\bf f}^n(t),{\bf 1}^n)_{\cH_n},\qquad t\in I_n;\qquad n\ge 0,
\ee
where $x_e(t)$ is the unique point on $e$ satisfying $|x_e(t)| = t$. 
Consider the map
\be\label{eq:Usym}
\begin{array}{cccc}
 U_s\colon & \cF_{\sym} &\to & L^2\big(\cI_\cL;\mu\big) \\
  & f & \mapsto & \tilde{f}
 \end{array}.
\ee
Clearly, for every $f\in \cF_{\sym}$, ${\bf f}^n(t) = \tilde{f}(t)\otimes {\bf 1}^n$ for a.e. $t\in I_n$ and hence
\[
\|\tilde{f}\|^2_{L^2(\cI_\cL;\mu)} = \sum_{n\ge 0} s_ns_{n+1}\|\tilde{f}\|^2_{L^2(I_n)} = \sum_{n\ge 0}\|{\bf f}^n\|^2_{L^2(I_n;\cH_n)} = \|{\bf f}\|^2 = \|f\|^2_{L^2(\cA)}.
\]
It turns out that
\be\label{eq:rH0equivFsym}
\rH_0 = U_s(\bH_0\upharpoonright \cF_{\sym})U_s^{-1}.
\ee
Indeed, $\bH_0$ acts as the negative second derivative on every edge $e\in\cE$ and hence for every $f\in\cF_{\sym}$ we get
\[
(U_s(\bH_0 f))(t) = - \tilde{f}''(t),\qquad t\in I_n,  
\]
for all $n\ge 0$. Therefore, it remains to show that $U_s(\cF_{\sym}\cap \dom(\bH_0)) = \dom(\rH_0)$. In fact, we only need to show that every $\tilde{f} = U_sf$ with $f\in \cF_{\sym}$ satisfies \eqref{eq:bcsym} if and only if $f\in \dom(\bH_0)$. Indeed, by \eqref{eq:tilde_f} and continuity of $f$, $\tilde{f}(t_n+) = \tilde{f}(t_{n}-)$ for all $n\ge 1$ if $f\in \cF_{\sym}\cap \dom(\bH_0)$. Moreover,  similar to \eqref{eq:3.05} one checks that 
\[
s_{n+1}\tilde{f}'(t_n+) = s_{n-1}\tilde{f}'(t_n-),\qquad n\ge0,
\]
exactly when $f\in \cF_{\sym}\cap \dom(\bH_0)$. This finishes the proof of Lemma \ref{lem:FsymH}.
\end{proof}

\subsection{Restriction to $\cF_n^0$}\label{ss:3.2}

Our next aim is to show that each $\cF_n^0$, $n \geq 1$, is a reducing subspace for $\bH_0$ and its restriction is unitarily equivalent to $(s_n-1) (s_{n+1}-1)$ copies of 
$\rh_n$, the second derivative with the Dirichlet boundary conditions on $L^2(I_n)$,
\be\label{eq:rh_n}
	\rh_n:= -\frac{d^2}{dt^2},\quad \dom(\rh_n) = \{f\in H^2(I_n)|\, f(t_{n}+)=f(t_{n+1}-)=0\}.
	\ee
By Lemma \ref{lem:Psym},  this will be a consequence of the following lemma.

\begin{lemma}\label{lem:F0restr}
	Let $n\ge 1$ be such that $s_n>1$ and $s_{n+1}>1$. Then each of the subspaces $\cF^n_{{\bf a}}$, where ${\bf a} = {\bf a}_n^{i,j}$ with $1\le i< s_n$ and $1\le j <s_{n+1}$, is reducing for the operator $\bH_0$. The restricted operator $\bH_0\upharpoonright \cF^n_{{\bf a}}$ is unitarily equivalent to the operator $\rh_n$ defined by \eqref{eq:rh_n}. 
\end{lemma}

\begin{proof}
Clearly, $\cF^n_{{\bf a}}$ is invariant for $\bH_0$. Since $\bH_0$ is symmetric, we only have to prove that $\ti{f}:=P^n_{{\bf a}} f \in \dom(\bH_0)$ whenever $f \in \dom(\bH_0)$. In fact, we need to show that $\ti{\bf f}:=U(P^n_{{\bf a}} f)$ given by \eqref{eq:P_an} satisfies conditions (i)--(iv) of Lemma \ref{lem:domH}. Conditions (i) and (ii) are obviously satisfied since $f\in \dom(\bH_0)$ and by the definition of $U(P^n_{{\bf a}}f)$. Since $\ti{\bf f}^m = 0$ for all $m\neq n$ and $n\ge 1$, (iii) clearly holds  and, moreover, we need to verify (iv) only at $t_n$ and $t_{n+1}$. 

Let us start with continuity. Suppose ${\bf a} = {\bf a}_n^{i,j}$ for some $1\le i< s_n$ and $1\le j <s_{n+1}$. First observe that
\[
\ti{\bf f}^n_{k,m}(t_n+) = \ti{\bf f}^n_{k,m}(t_{n+1}-) = 0 
\]
for all $k\in \{1,\dots,s_n\}$ and $m\in \{1,\dots,s_{n+1}\}$. Indeed,
\[
\lim_{t\to t_n+} ({\bf f}^n(t),{\bf a})_{\cH_n} = ({\bf f}^n(t_n+),{\bf a})_{\cH_n} = \sum_{k=1}^{s_n}  {\bf f}^n_{k,1}(t_n+) \omega_n^{-ik} \sum_{m=1}^{s_{n+1}}\omega_{n+1}^{-jm} = 0.
\]
Here we employed the continuity of $f$, ${\bf f}^n_{k,j}(t_n+) = {\bf f}^n_{k,1}(t_n+)$ for all $j\in \{1,\dots,s_{n+1}\}$, together with \eqref{eq:a_nij}. 
This shows that $\ti{\bf f}$ satisfies the first condition in (iv).

Next observe that 
\[
\sum_{m=1}^{s_{n+1}}(\ti{\bf f}^n_{k,m})'(t_n+) = \frac{\omega_n^{ik}}{s_ns_{n+1}}(({\bf f}^n)'(t_n+),{\bf a})_{\cH_n}\sum_{m=1}^{s_{n+1}} \omega_{n+1}^{jm} = 0
\]
for all $k\in\{1,\dots,s_n\}$. Since $(\ti{\bf f}^{n-1})'=0$, $\ti{\bf f}$ satisfies (iv) at $t_n$. Similar arguments shows that (iv) holds true at $t_{n+1}$ as well. This finishes the proof of the inclusion $\ti{f} = P^n_{{\bf a}} f \in \dom(\bH_0)$.

Finally, noting that 
\be\label{eq:Una}
\begin{array}{cccc}
 U_{\bf a}^n\colon & L^2(I_n) &\to & \cF^n_{\bf a} \\[1mm]
  & f & \mapsto & f\cdot \frac{{\bf a}}{\|{\bf a}\|} 
 \end{array}
\ee
establishes an isometric isomorphism of $L^2(I_n)$ onto $\cF_{\bf a}^n$, it is straightforward to verify the last claim and we left it to the reader.
\end{proof}

\subsection{Restriction to $\cF_n$}\label{ss:3.3}

Next, we show that $ \cF_n$, $n \geq 1$ is reducing for $\bH_0$ as well and the corresponding restriction is unitarily equivalent to $s_n-1$ copies of the operator $\wt{\rh}_n$ defined by
\[
\wt{\tau}_n = -\frac{1}{\mu(t)}\frac{d}{dt}\mu(t)\frac{d}{dt},
\]
on $L^2((t_{n-1},t_{n+1});\mu)$ and equipped with Dirichlet conditions at the endpoints. Here the weight function $\mu$ is defined by \eqref{eq:muL}. The domain of $\wt{\rh}_n$ admits a very simple description since inside $I_{n-1}$ and $I_n$ the differential expression $\wt{\tau}_n$ reduces to the negative second derivative and hence $\dom(\wt{\rh}_n)$ consists of functions which are $H^2$ in $I_{n-1}$ and $I_n$, satisfy the Dirichlet conditions at $t_{n-1}$ and $t_{n+1}$ and also the following coupling conditions at $t_n$:
\be \label{eq:bcrhn}
\begin{cases} 
f(t_n+)=f(t_n-) \\ s_{n-1} f'(t_n-) = s_{n+1} f'(t_n+)
\end{cases}.
\ee

Recall that $\cF_n = \ran (P_n)$, where the projection $P_n$ is given by \eqref{eq:Pn}. By \eqref{eq:a_nij} and \eqref{eq:1n}, 
\[
{\bf a}^{s_{n-1},j}_{n-1} = {\bf 1}_{s_{n-1}}\otimes {\bf a}_{s_n}^j,\qquad  {\bf a}_n^{j,s_{n+1}} = {\bf a}_{s_{n}}^j\otimes {\bf 1}_{s_{n+1}},
\]
and hence 
\be\label{eq:Pn2}
P_n = \sum_{j=1}^{s_n-1} \big(P^{n-1}_{{\bf 1}_{s_{n-1}}\otimes {\bf a}_{s_n}^j} + P^n_{{\bf a}_{s_{n}}^j\otimes {\bf 1}_{s_{n+1}}}\big).
\ee
Denoting the summands  in \eqref{eq:Pn2} by $\wt{P}_n^j$, $j\in \{1,\dots,s_n-1\}$, we set
\be
\wt{\cF}_n^j := \ran (\wt{P}_n^j) = \cF^{n-1}_{{\bf 1}_{s_{n-1}}\otimes {\bf a}_{s_n}^j}\oplus \cF^n_{{\bf a}_{s_{n}}^j\otimes {\bf 1}_{s_{n+1}}}.
\ee
Since $\cF_n	= \bigoplus_{j=1}^{s_n-1} \wt{\cF}_n^j$, these claims will follow from the following lemma:

\begin{lemma}\label{lem:wtFj}
Every subspace $\wt{\cF}_n^j$ with $n \geq 1$ and $j\in \{1,\dots,s_n-1\}$, is reducing for the operator $\bH_0$. Moreover, its restriction onto $\wt{\cF}_n^j$ is unitarily equivalent to $ \wt{\rh}_n$.
\end{lemma}

\begin{proof}
	Since $\wt{\cF}_n^j$ is invariant for $\bH_0$ and $\bH_0$ is symmetric, we only need to show that for every $f \in \dom(\bH_0)$ its projection $\ti{f} := \wt{P}_n^j f$ onto $\wt{\cF}_n^j$ also belongs to $\dom(\bH_0)$. Following step by step the proof of Lemma \ref{lem:F0restr}, we only need to show that $\ti{{\bf f}}:=U\ti{f}$ satisfies condition (iv) of Lemma \ref{lem:domH} at $t_n$. 
	
First observe that by \eqref{eq:P_an}
\be\label{eq:mapUn}
\ti{{\bf f}}(t) = \begin{cases} \ti{f}_{n-1}(t)({\bf 1}_{s_{n-1}}\otimes {\bf a}_{s_n}^j ), & t\in I_{n-1}\\
 \ti{f}_{n}(t)({\bf a}_{s_{n}}^j\otimes {\bf 1}_{s_{n+1}}), & t\in I_{n} \end{cases}
\ee
where
\[
\ti{f}_{n-1}(t) = \frac{1}{s_{n-1}s_{n}}({\bf f}^{n-1}(t), {\bf a}^{s_{n-1},j}_{n-1})_{\cH_{n-1}},\qquad \ti{f}_{n}(t) = \frac{1}{s_{n}s_{n+1}}({\bf f}^n(t), {\bf a}_n^{j,s_{n+1}})_{\cH_n}.
\]	 
Notice that 
\[
\ti{f}_{n-1}(t_n-) =  \frac{1}{s_{n-1}s_{n}}\sum_{k=1}^{s_{n-1}}\sum_{m=1}^{s_n} {\bf f}^{n-1}_{k,m}(t_n-)\omega_{n}^{-jm}=  \frac{1}{s_{n}}\sum_{m=1}^{s_n} {\bf f}^{n-1}_{1,m}(t_n-)\omega_{n}^{-jm}
\]
and
\[
\ti{f}_{n}(t_n+) =  \frac{1}{s_{n}s_{n+1}}\sum_{m=1}^{s_{n}}\sum_{k=1}^{s_{n+1}} {\bf f}^{n}_{m,k}(t_n+)\omega_n^{-jm}=  \frac{1}{s_{n}}\sum_{m=1}^{s_n} {\bf f}^{n}_{m,1}(t_n+)\omega_n^{-jm}.
\]
However, by Lemma \ref{lem:domH}, 
\[
{\bf f}^{n-1}_{1,m}(t_{n}-) = {\bf f}^n_{m,1}(t_n+),\qquad m\in \{1,\dots,s_n\}, 
\]
and hence we get
\begin{align*}
\ti{{\bf f}}^{n-1}_{1,k}(t_{n}-) = \frac{\omega_{n}^{jk}}{s_{n}}\sum_{m=1}^{s_n} {\bf f}^{n-1}_{1,m}(t_n-)\omega_{n}^{-jm} = \frac{\omega_{n}^{jk}}{s_{n}}\sum_{m=1}^{s_n} {\bf f}^{n}_{m,1}(t_n+)\omega_{n}^{-jm} = \ti{{\bf f}}^{n}_{k,1}(t_{n}+)
\end{align*}
for all $k\in\{1,\dots,s_n\}$. This shows that $\ti{{\bf f}}$ satisfies the first equality in condition (iv) of Lemma \ref{lem:domH}. Let us check the second one. However, we have
\begin{align*}
\sum_{k=1}^{s_{n-1}}(\ti{{\bf f}}^{n-1}_{k,m})'(t_n-) & =  \sum_{k=1}^{s_{n-1}}\ti{f}_{n-1}'(t_n-) \omega_n^{jm} = s_{n-1} \ti{f}_{n-1}'(t_n-) \omega_n^{jm} \\ & = \frac{\omega_n^{jm} }{s_{n}}\sum_{l=1}^{s_n}\omega_{n}^{-jl} \sum_{k=1}^{s_{n-1}}({\bf f}^{n-1}_{k,l})'(t_n-) = \frac{\omega_n^{jm} }{s_{n}}\sum_{l=1}^{s_n}\omega_{n}^{-jl} \sum_{k=1}^{s_{n+1}}({\bf f}^{n}_{l,k})'(t_n+) \\
& = s_{n+1} \ti{f}_{n}'(t_n+) \omega_n^{jm}=  \sum_{k=1}^{s_{n+1}}\ti{f}_{n}'(t_n+) \omega_n^{jm} = \sum_{k=1}^{s_{n+1}}(\ti{{\bf f}}^{n}_{m,k})'(t_n+).
\end{align*}
This shows that $\ti{{\bf f}}$ satisfies all the conditions of Lemma \ref{lem:domH} and hence $\ti{f}\in \dom(\bH_0)$. 

Finally, it is straightforward to check that the map $U_{n}^j\colon L^2((t_{n-1},t_{n+1});\mu) \to  \wt{\cF}^j_{n}$ defined by \eqref{eq:mapUn} is an isometric isomorphism and $(U_{n}^j)^{-1} (\bH_0 \upharpoonright \wt{\cF}^j_{n})U_{n}^j = \wt{\rh}_n$.\end{proof}

\subsection{The decomposition of the operator $\bH$}\label{ss:3.4}

Combining the results of Sections \ref{ss:3.1}--\ref{ss:3.3}, we arrive at the following decomposition of quantum graph operators on radially symmetric anti-trees.

\begin{theorem}\label{th:decomp}
Let $\cA$ be an infinite radially symmetric antitree. The decomposition \eqref{eq:basicedecomp} reduces the operator $\bH$. Moreover, with respect to this decomposition, $\bH$ is unitarily equivalent to the following orthogonal sum of Sturm--Liouville operators
\be\label{eq:Hdecomp}
\rH \oplus \bigoplus_{n\ge 1}\Big(\oplus_{j=1}^{(s_n-1)(s_{n+1}-1)} \rh_n\Big)\oplus\bigoplus_{n\ge 1}\Big(\oplus_{j=1}^{s_n-1} \wt{\rh}_n\Big),
\ee
where $\rH = \overline{\rH}_0$ and the operators $\rH_0$, $\rh_n$ and $\wt{\rh}_n$ are defined in Sections \ref{ss:3.1}, \ref{ss:3.2} and \ref{ss:3.3}, respectively.
\end{theorem}

\section{Self-adjointness}\label{sec:sa}

Theorem \ref{th:decomp} reduces the spectral analysis of quantum graph operators on radially symmetric antitrees to the analysis of certain classes of Sturm--Liouville operators. Moreover, the Sturm--Liouville operators $\rh_n$ and $\wt{\rh}_n$ in the decomposition \eqref{eq:Hdecomp} are self-adjoint for all $n\ge 1$ and their spectra can be computed explicitly. This enables us to perform a rather detailed study of spectral properties of the operator $\bH = \overline{\bH_0}$. We begin with the characterization of self-adjoint extensions of the operator $\bH$. 

\begin{theorem}\label{th:SA}
Let $\cA$ be an infinite radially symmetric antitree. 
Then:
\begin{enumerate}[label=(\roman*), ref=(\roman*), leftmargin=*, widest=iiii]
\item The operator $\bH$ is self-adjoint if and only if the total volume of $\cA$ is infinite,
\be\label{eq:volinfty}
\vol(\cA) := \sum_{e\in \cE(\cA)}|e| = \sum_{n\ge 0} s_n s_{n+1}\ell_n =\infty.
\ee
\item If $\vol(\cA)<\infty$, then the deficiency indices of $\bH$ equal $1$ and self-adjoint extensions of $\bH$ form a one-parameter family $\bH_\theta := \bH^\ast\upharpoonright \dom(\bH_\theta)$, where $\theta\in [0,\pi)$ and 
\begin{align*}
\dom(\bH_\theta):= \{f\in \dom(\bH^\ast)\colon \cos(\theta)(U_sP_{\sym} f)(\cL) + \sin(\theta)(U_sP_{\sym} f)_\mu'(\cL)= 0\}.
\end{align*}
The operators $P_{\sym}$ and $U_s$ are given, respectively, by \eqref{eq:Psym1} and \eqref{eq:Usym} and 
\begin{align}
(U_sP_{\sym} f)(\cL) & := \lim_{t\to \cL} (U_sP_{\sym} f)(t), \label{eq:bc01}
\\ (U_sP_{\sym} f)_\mu'(\cL) & := \lim_{t\to \cL} \mu(t)(U_sP_{\sym} f)'(t).\label{eq:bc02}
\end{align} 
\end{enumerate}
\end{theorem}

\begin{proof}
(i) By Theorem \ref{th:decomp}, the operator $\bH$ is self-adjoint only if so are the operators on the right-hand side in the decomposition \eqref{eq:Hdecomp}. However, both $\rh_n$ and $\wt{\rh}_n$ are self-adjoint for all $n\ge 1$. The self-adjointness criterion for $\rH = \overline{\rH_0}$ follows from the standard limit point/limit circle classification (see, e.g., \cite{wei}). Namely, the equation $\tau y = 0$ with $\tau$ given by \eqref{eq:tauA}, 
has two linearly independent solutions  
\begin{align*}
y_1(t) & \equiv 1, & y_2(t) & = \int_0^t \frac{ds}{\mu(s)}.
\end{align*}
Now one simply needs to verify whether or not both solutions $y_1$ and $y_2$ belong to $L^2(\cI_\cL;\mu)$. Clearly, $y_1\in L^2(\cI_\cL;\mu)$ exactly when 
the series in \eqref{eq:volinfty} converges. 
Moreover, it is straightforward to check that  
$y_2\in L^2(\cI_\cL;\mu)$ if and only if the series
\be\label{eq:saHX0}
\sum_{n\ge 0}s_n s_{n+1}\ell_n\Big( \sum_{k\le n}\frac{\ell_k}{s_k s_{k+1}}\Big)^2 
\ee
converges. Since $s_{n}s_{n+1}\ge 1$ for all $n\ge 0$, this series converges exactly when the series in \eqref{eq:volinfty} converges. The Weyl alternative finishes the proof of (i).

(ii) The above considerations imply that the deficiency indices of $\bH$ and $\rH$ coincide. However, the deficiency indices of $\rH$ are at most $1$. Thus, if the operator $\rH$ is not self-adjoint, its deficiency indices equal $1$. Moreover, one can easily describe all self-adjoint extensions of $\rH$. First of all, for every $f\in \dom(\rH_0^\ast) = \dom(\rH^\ast)$ the following limits
\begin{align*}
 \lim_{t\to\cL} & W_t(f,y_1), &  \lim_{t\to\cL} & W_t(f,y_2)
\end{align*}
exist and are finite (see, e.g., \cite{wei}). Here $W_t(f,g) = f(t) (\mu g')(t)- (\mu f')(t)g(t)$ is the modified Wronskian.  
Thus  for every $f\in \dom(\rH_0^\ast)$ the following limits
\begin{align}
f(\cL) & := \lim_{t\to \cL} f(t), & f_\mu'(\cL) & := \lim_{t\to\cL}\mu(t)f'(t)
\end{align}
exist and are finite. Hence self-adjoint extensions of $\rH$ form a one-parameter family 
\[
\dom(\rH(\theta)) := \big\{ f\in \dom(\rH_0^\ast)|\, \cos(\theta)f(\cL) + \sin(\theta)f_\mu'(\cL) = 0\big\},\quad \theta\in [0,\pi).
\]
It remains to use \eqref{eq:rH0equivFsym} and  \eqref{eq:Psym1}.
\end{proof}

\begin{remark}
Let us mention that in the case $\vol(\cA)<\infty$ the Friedrichs extension of $\bH$ coincides with the operator $\bH_\theta$ with $\theta=0$. Moreover, it is possible to show that in fact the limits in \eqref{eq:bc01} and \eqref{eq:bc02} coincide with 
\begin{align*}
\lim_{|x|\to \cL} & f(x), & \lim_{t\to \cL} & \sum_{|x|=t}f'(x)
\end{align*}
for every $f$ in the domain of $\bH^\ast$. In particular, this would imply that the Friedrichs extension of $\bH$ is simply given as the restriction of $\bH^\ast$ to functions vanishing at $\cL$. Let us also mention that $\bH^\ast = \bH_0^\ast$ in fact coincides with the maximal operator, that is $\dom(\bH^\ast)$ consists of functions $f\in L^2(\cA)\cap H^2(\cA\setminus \cV)$ satisfying boundary conditions \eqref{eq:kirchhoff} for all $v\in \cV$ and such that $f'' \in L^2(\cA)$. 
\end{remark}

\section{Discreteness}\label{sec:discr}
As an immediate corollary of Theorem \ref{th:SA} we obtain the following result.

\begin{corollary}\label{cor:finvol}
If $\vol(\cA)<\infty$, then the spectrum of each self-adjoint extension $\bH_\theta$ of $\bH$ is purely discrete and, moreover, 
\be\label{eq:WeylLaw}
N(\lambda;\bH_\theta) = \frac{\vol(\cA)}{\pi}\sqrt{\lambda} (1+o(1)),\quad \lambda\to \infty,
\ee
for all $\theta\in [0,\pi)$.
\end{corollary}

Here $N(\lambda;A)$ is the eigenvalue counting function of a (bounded from below) self-adjoint operator  $A$ with purely discrete spectrum. Namely,
\[
N(\lambda;A) = \# \{k\colon \lambda_k(A) \le \lambda\},
\]
where $\{\lambda_k(A)\}_{k\ge 0}$ are the eigenvalues of $A$ (counting multiplicities) ordered in the increasing order.

\begin{proof}
By Theorem \ref{th:decomp}, 
\be
\sigma(\bH_\theta) = \sigma(\rH(\theta)) \cup \ol{ \cup_{n\ge 1}\sigma({\rh}_n)} \cup \ol{ \cup_{n\ge 1}\sigma(\wt{\rh}_n)}.
\ee
Since $s_n\ge 1$ for all $n\ge 1$, $\vol(\cA)<\infty$ implies that $\ell_n = o(1)$ as $n\to \infty$ and hence both sets $\cup_{n\ge 1}\sigma({\rh}_n)$ and $\cup_{n\ge 1}\sigma(\wt{\rh}_n)$ have no finite accumulation points. It remains to note that the spectrum of $\rH(\theta)$ is discrete in this case as well.

According to the decomposition \eqref{eq:Hdecomp}, we clearly have
\[
N(\lambda;\bH_\theta) = N(\lambda;\rH(\theta)) + \sum_{n\ge 1} (s_{n}-1)(s_{n+1}-1)N(\lambda;\rh_n) + \sum_{n\ge 1}(s_{n}-1)N(\lambda;\wt{\rh}_n).
\]
It is well known that (cf., e.g., \cite[Chapter 6.7]{gohkre})
\[
N(\lambda;\rH(\theta)) = \frac{\cL}{\pi}\sqrt{\lambda} (1+o(1)),\quad \lambda\to \infty,
\]
for all $\theta\in[0,\pi)$. 
Taking into account that 
\begin{align}\label{eq:sigma_wthn}
N(\lambda;\rh_n) & = \frac{\ell_n}{\pi}\sqrt{\lambda} (1+o(1)), & N(\lambda;\wt{\rh}_n) & = \frac{\ell_{n-1}+\ell_{n}}{\pi}\sqrt{\lambda} (1+o(1)),
\end{align}
we immediately arrive at \eqref{eq:WeylLaw}.
\end{proof}

\begin{remark}\label{rem:sigma_n}
Recall that 
\be\label{eq:sigma_hn}
\sigma({\rh}_n) = \left\{\frac{\pi^2k^2}{\ell_n^2}\right\}_{k\ge 1}.
\ee
However, we are not aware (except a few special cases) of a closed form of eigenvalues of $\wt{\rh}_n$.  It is not difficult to show that $\sigma(\wt{\rh}_n)$ consists of simple positive eigenvalues $\{\wt{\lambda}_k\}_{k\ge 1}$ satisfying \eqref{eq:sigma_wthn} and even to express $\sigma(\wt{\rh}_n)$ with the help of the {\em arctangent function with two arguments}, \iflong{}(see Appendix \ref{app:atan2})\fi although this does not lead to a closed formula. 
\end{remark}

In the case $\vol(\cA)=\infty$, the spectrum of $\bH$ may have a rather complicated structure. In particular, it may not be purely discrete. The next result provides a criterion for $\bH$ to have purely discrete spectrum. Set
\be\label{eq:Lmu}
\cL_\mu := \int_0^{\cL}\frac{d x}{\mu(x)} = \sum_{n\ge 0} \frac{\ell_n}{s_ns_{n+1}}.
\ee

\begin{theorem}\label{th:discr}
Let $\cA$ be an infinite radially symmetric antitree with $\vol(\cA)=\infty$.  
Then the spectrum of $\bH$ is discrete if and only if the following conditions are satisfied:
\begin{itemize}
\item[{\rm (i)}] $\ell_n\to 0$ as $n\to \infty$, 
\item[{\rm (ii)}] $\cL_\mu <\infty$, and
\item[{\rm (iii)}] 
\be\label{eq:KK07}
\lim_{n\to \infty}\sum_{k=0}^n s_ks_{k+1}\ell_k \sum_{k\ge n} \frac{\ell_k}{s_ks_{k+1}} = 0.
\ee
\end{itemize}
\end{theorem}

\begin{proof}
Denote 
\begin{align}\label{eq:summands}
\bH^1 & :=\bigoplus_{n\ge 1}\Big(\oplus_{j=1}^{(s_n-1)(s_{n+1}-1)} \rh_n\Big), & 
\bH^2 & := \bigoplus_{n\ge 1}\Big(\oplus_{j=1}^{s_n-1} \wt{\rh}_n\Big).
\end{align}
By Theorem \ref{th:SA}(i), $\bH$ is self-adjoint and hence \eqref{eq:Hdecomp} implies that
\be\label{eq:bHspec}
\sigma(\bH) = \sigma(\rH) \cup \sigma(\bH^1)\cup\sigma(\bH^2) = \sigma(\rH) \cup \ol{ \big(\cup_{n\ge 1}\sigma({\rh}_n)\big)} \cup \ol{ \big(\cup_{n\ge 1}\sigma(\wt{\rh}_n)\big)}.
\ee
Thus the spectrum of $\bH$ is discrete if and only if the spectra of all three operators $\rH$, $\bH^1$ and $\bH^2$ are discrete. 

In order to investigate the operator $\rH$, we  need to transform it to the Krein string form by using a suitable change of variables ($x\mapsto \int_0^x \frac{ds}{\mu(s)}$) and then to apply the Kac--Krein criterion \cite{kakr58}. To be more precise,
 it is straightforward to verify that $\rH$ is unitarily equivalent to the operator $\wt{\rH}$ defined in the Hilbert space $L^2([0,\cL_\mu);\wt\mu)$
by the differential expression
\be\label{eq:wtH}
\wt{\tau} = -\frac{1}{\wt\mu(x)}\frac{d^2}{d x^2}
\ee
and subject to the Neumann boundary condition at $x=0$. Here 
\be\label{eq:mu_g}
\wt\mu := \mu^2\circ g^{-1},
\ee
where $g^{-1}$ is the inverse of the function $g\colon [0,\cL)\to [0,\cL_\mu) $ given by
\begin{align}\label{eq:g_def}
g(x) & = \int_0^x \frac{ds}{\mu(s)}, & \cL_\mu & := g(\cL) = \int_0^\cL \frac{ds}{\mu(s)} .
\end{align}
Notice that $g$ is strictly increasing and locally absolutely continuous on $[0,\cL)$ and maps $[0,\cL)$ onto $[0,\cL_\mu)$. Hence its inverse $g^{-1}\colon [0,\cL_\mu)\to [0,\cL)$ is also strictly increasing and locally absolutely continuous on $[0,\cL_\mu)$.

Applying the Kac--Krein criterion (see \cite{kakr58}, \cite[\S 11.9]{kakr74}), we conclude that $\rH$ has purely discrete spectrum if and only if $\cL_\mu<\infty$
and
\be\label{eq:KK06}
\lim_{x\to\cL} \Phi(x) = 0,
\ee
where $\Phi\colon [0,\cL)\to \R_{\ge 0}$ is given by 
\be\label{eq:cM}
\Phi(x):= \int_0^x\mu(s)ds\cdot\int_x^{\cL}\frac{ds}{\mu(s)},\quad x\in[0,\cL).
\ee
First of all, observe that
\[
\Phi(x) \le \int_0^{t_{n+1}}\mu(s)ds\cdot\int_{t_n}^{\cL}\frac{ds}{\mu(s)} = \sum_{k=0}^n s_ks_{k+1}\ell_k \sum_{k\ge n} \frac{\ell_k}{s_ks_{k+1}}
\]
for all $x\in [t_n,t_{n+1})$ and hence sufficiency of  \eqref{eq:KK07} follows. 
Moreover, straightforward calculations show that
\begin{align*}
   \Phi\Big(\frac{t_n+t_{n+1}}{2}\Big)
   & = \Big(\sum_{k=0}^{n-1} s_ks_{k+1}\ell_k + s_ns_{n+1}\frac{\ell_n}{2}\Big)
         \Big(\sum_{k\ge n+1} \frac{\ell_k}{s_ks_{k+1}} + \frac{\ell_{n}}{2s_ns_{n+1}}\Big) \\
   & \ge \frac{1}{4}\sum_{k=0}^n s_ks_{k+1}\ell_k \sum_{k\ge n} \frac{\ell_k}{s_ks_{k+1}},
\end{align*}
which implies the necessity of \eqref{eq:KK07}. Notice also that the right-hand side in the last inequality is strictly greater than $\frac{1}{4}\ell_n^2$, which also implies (i). 

It remains to note that the spectra of the operators $\bH^1$ and $\bH^2$
 are discrete if condition (i) is satisfied (see \eqref{eq:sigma_hn} and \eqref{eq:sigma_wthn}). 
\end{proof}

\begin{remark}\label{rem:discr}
Let us mention that in fact both conditions (i) and (ii) in Theorem \ref{th:discr} follow from (iii).
\end{remark}

If $\vol(\cA) =\infty$ and the corresponding Hamiltonian $\bH$ has purely discrete spectrum, it follows from the proof of Weyl's law \eqref{eq:WeylLaw} that $\frac{N(\lambda;\bH)}{\sqrt{\lambda}} \to \infty$ as $\lambda\to \infty$. However, we can characterize radially symmetric antitress such that the resolvent of the corresponding quantum graph operator $\bH$ belongs to the trace class.

\begin{theorem}\label{th:trace}
Let $\cA$ be an infinite radially symmetric antitree with $\vol(\cA)=\infty$. Also, let the spectrum of $\bH$ be purely discrete. 
Then\footnote{The summation in \eqref{eq:Htrace} is according to multiplicities.}
\be\label{eq:Htrace}
\sum_{\lambda \in \sigma(\bH)} \frac{1}{\lambda} <\infty
\ee
if and only if 
\be\label{eq:trace01}
\sum_{n\ge 1} s_ns_{n+1}\ell_n^2 <\infty,
\ee
and
\be\label{eq:KK08}
 \sum_{n\ge 0}\frac{\ell_n}{s_ns_{n+1}} \sum_{k=0}^{n-1}s_ks_{k+1}\ell_k <\infty.
\ee
\end{theorem}

\begin{proof}
As in the proof of Theorem \ref{th:discr}, observe that the spectrum of $\bH$ consists of three sets of eigenvalues. Let us denote the second and the third summands in \eqref{eq:Hdecomp} by $\bH^1$ and $\bH^2$, respectively. The spectrum of the self-adjoint operator $\rh_n$ is given by \eqref{eq:sigma_hn} and hence
\be\label{eq:trace02}
\sum_{\lambda\in \sigma(\bH^2)} \frac{1}{\lambda} = \sum_{n\ge 1} (s_n-1)(s_{n+1}-1)\sum_{k\ge 1}\frac{\ell_n^2}{\pi^2k^2} = \frac{1}{6}\sum_{n\ge 1} (s_n-1)(s_{n+1}-1)\ell_n^2.
\ee
 Similarly, the spectrum of the self-adjoint operator $\wt{\rh}_n$ also consists of simple positive eigenvalues, however, we are not aware of their closed form. Instead one can employ the standard Dirichlet--Neumann bracketing, that is, to estimate the eigenvalues of $\wt{\rh}_n$ via the eigenvalues of the operators $\wt{\rh}_n^D$ and $\wt{\rh}_n^N$ subject to Dirichlet, respectively, Neumann boundary conditions at $t_n$:
 \[
 \lambda_k(\wt{\rh}_n^N) \le  \lambda_k(\wt{\rh}_n) \le  \lambda_k(\wt{\rh}_n^D)
 \]
 for all $k\ge 1$. Thus, we get
\begin{align*}
\sum_{\lambda\in \sigma(\bH^1)} \frac{1}{\lambda} & \le \sum_{n\ge 1} (s_n-1)\sum_{\lambda\in \sigma(\wt{\rh}_n^N)} \frac{1}{\lambda} \\
&= \sum_{n\ge 1} (s_n-1)\sum_{k\ge 1}\frac{\ell_{n-1}^2}{\pi^2(k-1/2)^2} + \frac{\ell_{n}^2}{\pi^2(k-1/2)^2} \\
&  = \frac{1}{2}\sum_{n\ge 1} (s_n-1)(\ell_{n-1}^2+\ell_{n}^2) \le \frac{1}{2}\sum_{n\ge 0} (s_n+s_{n+1}-2)\ell_{n}^2.
\end{align*}
Using the Dirichlet eigenvalues, one can prove a similar bound from below. Moreover, combining the latter with \eqref{eq:trace01} implies that the resolvents of both $\bH_1$ and $\bH_2$ belong to the trace class exactly when 
\be\label{eq:trace02B}
\sum_{n\ge 1} (s_ns_{n+1}-1)\ell_n^2 <\infty.
\ee

Next observe that $0\in \sigma(\rH)$ exactly when $\id \in L^2(\cI_\cL;\mu)$, which is equivalent to $\vol(\cA)<\infty$. Thus $0$ is not an eigenvalue of $\bH$ whenever $\vol(\cA)=\infty$. Finally, applying M.\ G.\ Krein's theorem to the operator $\rH$ (see \cite{kakr58}, \cite[\S 11.10]{kakr74}), we conclude that $\rH^{-1}$ is trace class if and only if $\cL_\mu<\infty$ and 
\be\label{eq:KK09}
\int_0^\cL \frac{1}{\mu(x)} \int_0^x {\mu(s)}{ds}\, dx <\infty.
\ee
However, using \eqref{eq:muL}, we get
\begin{align*}
\int_0^\cL \frac{1}{\mu(x)} \int_0^x {\mu(s)}{ds}\, dx & = \sum_{n\ge 0} \int_{t_n}^{t_{n+1}} \frac{1}{\mu(x)} \int_0^x {\mu(s)}{ds}\, dx \\
& = \sum_{n\ge 0}\frac{1}{s_ns_{n+1}} \int_{t_n}^{t_{n+1}}  \Big(\sum_{k=0}^{n-1}s_ks_{k+1}\ell_k + s_ns_{n+1}(x-t_n) \Big)\, dx \\
& = \sum_{n\ge 0}\frac{\ell_n}{s_ns_{n+1}} \sum_{k=0}^{n-1}s_ks_{k+1}\ell_k + \frac{1}{2}\sum_{n\ge 0}\ell_n^2.
\end{align*}
Notice that the latter in particular shows that $\{\ell_n\}_{n\ge 0}\in \ell^2$ and combining this fact with \eqref{eq:trace02B} we arrive at \eqref{eq:trace01}. This completes the proof of Theorem \ref{th:trace}.
\end{proof}

\begin{remark}
Using the same arguments and the results from \cite{kac, rowo19} one would be able to characterize radially symmetric antitrees such that the resolvent of the corresponding Kirchhoff Laplacian belongs to the Schatten--von Neumann ideal $\mathfrak{S}_p$, $p\in (1,\infty)$ (and even to  other trace ideals), however, these results look cumbersome and we decided not to include them.
\end{remark}

\section{Spectral gap estimates}\label{sec:specest}
We restrict our discussion to the case $\vol(\cA) = \infty$ for several reasons. Of course, the main one  is the fact that in this case $\bH_0$ is essentially self-adjoint and this simplifies some considerations. However, for finite volume metric graphs the corresponding estimates remain to be true for the Friedrichs extension of $\bH_0$. 

Our next goal is to estimate the bottom of the spectrum of the operator $\bH$. 

\begin{theorem}\label{th:SpEst}
Let $\cA$ be an infinite radially symmetric antitree with $\vol(\cA)=\infty$.   
Then the bottom of the spectrum $\lambda_0(\bH) = \inf \sigma(\bH)$ of $\bH$ is strictly positive if and only 
if the following conditions are satisfied:
\begin{itemize}
\item[{\rm (i)}] $\ell^\ast(\cA) = \sup_{n\ge 0}\ell_n<\infty$,
\item[{\rm (ii)}] $\cL_\mu <\infty$, and
\item[{\rm (iii)}] 
\be\label{eq:KK04}
C(\cL):=\sup_{x\in(0,\cL)}\int_0^x \mu(s)ds\cdot\int_x^{\cL}\frac{ds}{\mu(s)}  <\infty.
\ee
Moreover, we have the following estimate
\be\label{eq:estKK01}
\frac{1}{4C(\cL)} \le \lambda_0(\bH) \le \frac{1}{C(\cL)}.
\ee
\end{itemize}
\end{theorem}

\begin{proof}
Since $\vol(\cA)=\infty$, the operator $\bH$ is self-adjoint by Theorem \ref{th:SA}. Moreover, by Theorem \ref{th:decomp}, we have
\be
\lambda_0(\bH) = \min\{\lambda_0(\rH), \lambda_0(\bH^1),\lambda_0(\bH^2)\},
\ee
where $\bH^1$ and $\bH^2$ are given by \eqref{eq:summands}. 
Observe that 
\be\label{eq:estbH=rH}
\lambda_0(\bH) = \lambda_0(\rH).
\ee
Indeed, it suffices to compare the domains of $\rH_0$ and $\rh_n$, $\wt{\rh}_n$ and then exploit the Rayleigh quotient. For instance, 
\begin{align*}
\lambda_0(\rH) & = \inf_{\substack{f\in\dom(\rH_0)\\ f\neq0}} \frac{(\rH f,f)_{L^2(\cI_\cL;\mu)}}{\|f\|^2_{L^2(\cI_\cL;\mu)}}  \le \inf_{\substack{f\in\dom(\rH_0)\\ {\rm supp}(f)\subset [t_{n-1},t_{n+1}]}} \frac{(\rH f,f)_{L^2(\cI_\cL;\mu)}}{\|f\|^2_{L^2(\cI_\cL;\mu)}} \\
& \le \inf_{\substack{f\in\dom(\wt{\rh}_n)\\ f\neq 0}} \frac{(\wt{\rh}_n f,f)_{L^2(I_{n-1} \cup I_{n};\mu)}}{\|f\|^2_{L^2(I_{n-1}\cup I_{n};\mu)}} = \lambda_0(\wt{\rh}_n).
\end{align*}
 
 The operator $\rH$ can be studied in the framework of Krein strings, however, we need to apply the Kac--Krein criteria \cite{kakr58} to the {\em dual string} since both Corollary 1.1 and Remark 2.2 in \cite{kakr58} are stated subject to the Dirichlet boundary condition at $x=0$. For a detailed discussion of dual strings we refer to \cite[\S 12]{kakr74} and the desired connection is \cite[equality (12.6)]{kakr74}
 \footnote{This statement can be seen as the analog of the {\em abstract commutation}: for a closed operator $A$ acting in a Hilbert space $\gH$, the operators $(A^\ast A)\upharpoonright_{\ker(A)^\perp} $ and $(AA^\ast) \upharpoonright_{\ker(A^\ast )^\perp} $ are unitarily equivalent.}. 
 More precisely, assuming that $\cL_\mu<\infty$ and then applying Theorem 1 from \cite{kakr58}, we get the estimate
\be\label{eq:kk01}
x\left( M^{-1}(\infty) - M^{-1}(x)\right) \le \frac{1}{\lambda_0(\rH)},
\ee
which holds for all $x>0$. Here $M^{-1}$ denotes the inverse to the function $M\colon [0,\cL_\mu)\to [0,\infty)$ defined by (see also \eqref{eq:mu_g} and \eqref{eq:g_def})
\be\label{eq:defM}
M(x):= \int_0^x \wt{\mu}(s)ds = \int_0^x (\mu^2\circ g^{-1})(s)ds = \int_0^{g^{-1}(x)} \mu(s)ds.
\ee
Notice that $M$ is a strictly increasing absolutely continuous function mapping $[0,\cL_\mu)$ onto $[0,\infty)$ (the latter follows from the assumption $\vol(\cA)=\infty$).  Thus \eqref{eq:kk01} is equivalent to 
\be
M(x)\left( \cL_\mu - x\right) \le \frac{1}{\lambda_0(\rH)},\qquad x\in (0,\cL_\mu).
\ee
By changing variables, we end up with the following estimate
\be
\sup_{x\in(0,\cL)}\int_0^x\mu(s)ds\cdot\int_x^{\cL}\frac{ds}{\mu(s)} \le \frac{1}{\lambda_0(\rH)}.
\ee
 Applying Theorem 3 from \cite{kakr58} and using the same arguments, we end up with the lower bound
\be
 \frac{1}{4\lambda_0(\rH)} \le\sup_{x\in (0,\cL)} \int_0^x\mu(s)ds\cdot\int_x^{\cL}\frac{ds}{\mu(s)}.
\ee
Taking into account \cite[Remark 2.2]{kakr58}, we conclude that the condition $\cL_\mu<\infty$ is also necessary for the positivity of $\lambda_0(\rH)$. It remains to note that the necessity of (i) follows from (iii). Indeed, assuming the converse, that is, there is a sequence of lengths $\ell_{n_k}$ tending to infinity, and then choosing $x_{n_k}$ as the middle points of the corresponding intervals, one immediately concludes that $C(\cL) = \infty$ by evaluating \eqref{eq:KK04} at $x_{n_k}$.
\end{proof}

\begin{remark}\label{rem:SpEst01}
Arguing as in the proof of Theorem \ref{th:discr} one can show that conditions (i)--(iii) in Theorem \ref{th:SpEst} can be replaced by the single condition
\be\label{eq:Iso01}
\sup_{n\ge 0}\, \sum_{k=0}^n s_ks_{k+1}\ell_k \sum_{k\ge n} \frac{\ell_k}{s_ks_{k+1}} <\infty.
\ee
However, this expression provides only an upper bound on $C(\cL)$:
\be\label{eq:Iso02}
\sup_{n\ge 0}\,\sum_{k=0}^{n} s_ks_{k+1}\ell_k \sum_{k\ge n+1} \frac{\ell_k}{s_ks_{k+1}} \le C(\cL) \le \sup_{n\ge 0}\,\sum_{k=0}^n s_ks_{k+1}\ell_k \sum_{k\ge n} \frac{\ell_k}{s_ks_{k+1}}.
\ee
\end{remark}

Since $0$ is not an eigenvalue of $\bH$ if $\vol(\cA)=\infty$, $\lambda_0(\bH)>0$ is equivalent to $\lambda_0^{\ess}(\bH)>0$, where $\lambda_0^{\ess}(\bH)$ denotes the bottom of the essential spectrum of $\bH$, $ \lambda_0^{\ess}(\bH): = \inf \sigma_{\ess}(\bH)$. Thus Theorem \ref{th:SpEst} also provides a criterion for $\lambda_0^{\ess}(\bH)$ to be strictly positive. Moreover, by employing  Glazman's decomposition principle one can prove a similar to \eqref{eq:KK04} bound on $\lambda_0^{\ess}(\bH)$.

\begin{theorem}\label{th:SpEstEss}
Let $\cA$ be an infinite radially symmetric antitree with $\vol(\cA)=\infty$.   
Then $\lambda_0^{\ess}(\bH)>0$  if and only if \eqref{eq:Iso01} holds true. Moreover, 
\be\label{eq:estKK02}
\frac{1}{4C_{\ess}(\cL)} \le \lambda_0^{\ess}(\bH) \le \frac{1}{C_{\ess}(\cL)},
\ee
where the constant $C_{\ess}(\cL)$ is given by
\be
C_{\ess}(\cL) = \lim_{x\to \cL}\sup_{y\in(x,\cL)}\int_x^y \mu(s)ds\cdot\int_y^{\cL}\frac{ds}{\mu(s)}.
\ee
\end{theorem}

A few remarks are in order.

\begin{remark}\label{rem:SpEst03}
\begin{itemize}
\item[(i)] Notice that the equality $C_{\ess}(\cL) = 0$ implies Theorem \ref{th:discr}.
\item[(ii)] One can prove Theorem \ref{th:SpEst} avoiding the use of the Kac--Krein results \cite{kakr58}. Namely, with the help of the Rayleigh quotient, one can rewrite the inequality $\lambda_0(\rH)>0$ as a variational problem and then apply Muckenhoupt's inequalities (see, e.g., \cite[\S 1.3.1]{maz}, \cite{muc}). In particular, M.\ Solomyak employed this approach in the study of quantum graph operators on radially symmetric trees (see \cite[\S 5]{sol04}). 
\item[(iii)] It is interesting to compare Theorems \ref{th:SpEst} and \ref{th:SpEstEss} with volume growth estimates (cf. \cite{stu}). For instance, by \cite[Theorem 7.1]{kn19}, 
\be\label{eq:volgr}
\lambda_0(\bH) \le \lambda_0^{\ess}(\bH) \le \frac{1}{4}{\rm v}(\cA)^2,
\ee
where
\be
{\rm v}(\cA):=\liminf_{n\to \infty} \frac{1}{\sum_{k=0}^n \ell_k}\log\Big(\sum_{k=0}^n s_ks_{k+1}\ell_k\Big).
\ee
However, this result applies only if $\cL=\sum_{n\ge 0} \ell_n=\infty$. 
\end{itemize}
\end{remark}

\section{Isoperimetric constant}\label{app:iso}

Recall that the {\em isoperimetric constant} $\alpha (\cG)$ of a metric graph $\cG$ is  (see \cite[\S 3]{kn19})
\begin{equation} \label{eq:defalpha}
	\alpha(\cG) := \inf_{\wt \cG} \frac{\deg_{\cG} (\partial \wt \cG)}{\vol(\wt \cG)},
\end{equation}
where the infimum is taken over all finite connected subgraphs $\wt \cG = (\wt \cV, \wt \cE)$. Here
\begin{align*}
	\partial \wt \cG = \{ v \in \wt \cV | \; \deg_{\wt \cG} (v) < \deg_{\cG} (v) \},
\end{align*}
is the {\em boundary} of $\wt \cG$ and
\begin{align}
	\deg_{\cG} (\partial \wt \cG) & := \sum_{v \in \partial \wt \cG} \deg_{\wt \cG} (v), & 	\vol(\wt \cG) & := \sum_{e \in \wt \cE} |e|.
\end{align}

Computation of the isoperimetric constant is known to be an NP-hard problem, however, due to the presence of symmetries, we are able to find $\alpha(\cA)$ for radially symmetric antitrees.

\begin{theorem} \label{th:alpha}
The isoperimetric constant of a radially symmetric antitree  $\cA$ is 
\begin{equation} \label{eq:alpha}
	{\alpha(\cA)} = \inf_{n \ge 0} \frac{s_n s_{n+1}}{ \sum_{k =0}^n s_k s_{k+1} \ell_k}.
\end{equation}
\end{theorem}

\begin{proof}
The decomposition obtained in Theorem \ref{th:decomp} suggests to take the infimum in \eqref{eq:defalpha} only over radially symmetric subgraphs. Namely, choosing $\cA_n $ for every $n\ge 0$ as the subgraph consisting of all edges between the root $o$ and the combinatorial sphere $S_{n+1}$, we have $\partial  \cA_n = S_{n+1}$ and $\deg_{\cA_n} (v) = s_n$ for all vertices $v \in S_{n+1}$. Hence by \eqref{eq:defalpha} we get
	\be\label{eq:alpha_upper}
		\alpha(\cA) \le \frac{\deg ( \partial \cA_n) }{\vol(\cA_n)} = \frac{s_n s_{n+1}}{\sum_{k \le n} s_k s_{k+1} \ell_k}.
	\ee
Thus it remains to show that indeed it suffices to restrict the infimum in \eqref{eq:defalpha} to the family $\{\cA_n\}_{n\ge0}$. Observe that $\{\cA_n\}_{n\ge 0}$ is a net, that is, for every finite connected subgraph $\wt\cA$ of $\cA$ there is $n\ge 0$ such that $\wt\cA$ is a subgraph of $\cA_n$. Hence we will proceed by induction in $n$. 

Let us start with subgraphs $\wt \cA\subsetneq\cA_0$. Then $\wt \cA$ consists of $m < s_0s_1 $ edges of $\cE_0^+$ and $\vol(\wt \cA) = m \ell_0$. Moreover, for all vertices of $\wt\cA$, $\deg_{\wt\cA}(v) <\deg_\cA(v)$ and hence $\deg(\partial{\wt \cA}) = 2m$, which implies
	\[
		\frac{ \deg( \partial \wt  \cA)}{\vol (\wt \cA)} =  \frac{2m}{m\ell_0} = \frac{2}{ \ell_0} > \frac{ \deg( \partial\cA_0)}{\vol (\cA_0)} = \frac{1}{\ell_0}.
	\]
	
Take $n\ge 1$ and assume that 
\be\label{eq:iso_n}
\frac{ \deg( \partial \wt  \cA)}{\vol (\wt \cA)} \ge \inf_{k\le n-1}\frac{ \deg( \partial\cA_k)}{\vol (\cA_k)} = \inf_{k\le n-1}\frac{s_k s_{k+1}}{\sum_{j \le k} s_j s_{j+1} \ell_j}
\ee
holds for all connected subgraphs $\wt\cA\subseteq \cA_{n-1}$. Take now a connected subgraph $\wt\cA \subseteq \cA_{n}$ such that $\wt\cA \not\subseteq \cA_{n-1}$. The latter in particular implies that $\cV(\wt\cA)\cap S_{n} \neq \emptyset$ and $\cV(\wt\cA)\cap S_{n+1} \neq \emptyset$. We can also assume that $\cV(\wt\cA)\cap S_{n-1} \neq \emptyset$ since otherwise 
$\cE(\wt\cA)\subseteq \cE_n^+$ and hence in this case
\be\label{eq:iso_ln}
\frac{ \deg( \partial \wt  \cA)}{\vol (\wt \cA)} = \frac{2}{\ell_n} > \frac{s_n s_{n+1}}{\sum_{k \le n} s_k s_{k+1} \ell_k} = \frac{\deg ( \partial \cA_n) }{\vol(\cA_n)}. 
\ee
Let us first show that without loss of generality we can take $\wt{\cA}$ such that each edge $e \in \cE(\wt\cA)$ contains at least one vertex in $\cV_{\rm int}(\wt\cA):=\cV(\wt{\cA})\setminus\partial \wt{\cA}$. 
Indeed, if not, consider the induced subgraph $\wt\cA_{\rm int}$, which we can split into a finite disjoint union of connected subgraphs $\{\wt{\cA}_j\}$. In particular, $ \wt \cV_{\rm int} = \cup_{j} \cV(\wt{\cA}_j)$. Let $ \cG_j$ be the star-like subgraphs of $\cA$ with edge sets $\cE(\cG_j) = \cup_{v \in \cV(\wt{\cA}_j)} \cE_v$. By construction, $\cG_j \subseteq \cA_n$ and each edge of $ \cG_j$ contains a vertex from $\cV(\cG_j)\setminus\partial\cG_j = \cV(\wt{\cA}_j)$. Moreover, let $\cE_r = \cE(\wt{\cA}) \setminus \cup_{j} \cE(\cG_j)$ be the remaining edges of $\wt{\cA}$. Then 
it is straightforward to verify (see also \cite[proof of Lemma 3.5]{n18}) that
	\begin{align*}
		\frac{ \deg(\partial \wt{\cA} )}{\vol(\wt \cA)} = \frac { \sum_j \deg(\partial \cG_j) + 2 \# \cE_r } { \sum_j \vol(\cG_j)  + \sum_{e \in  \cE_r} |e|} \geq \min_{j,e\in\cE_r} \left 	\{  \frac{ \deg(\partial \cG_j )}{ \vol(\cG_j)}, \frac{2}{ | e |} \right \}.
	\end{align*}
Taking into account \eqref{eq:iso_ln}, this proves the claim.

Consider a new graph $\wt{\cA}'$ obtained from $\wt\cA$ by adding all possible edges connecting $S_n$ with $S_{n-1}$ and $S_{n+1}$ such that the new graph $\wt{\cA}'$ is connected. By construction, $\wt{\cA}' \subseteq \cA_n$. Moreover, $S_{n+1}\subseteq \partial \wt{\cA}'$ and $\deg_{\wt{\cA}'}(v) = s_n$ for all $v\in S_{n+1}$. Hence
\[
\frac{ \deg( \partial \wt{\cA}')}{\vol (\wt{\cA}')} \ge \frac{s_ns_{n+1}}{\vol(\cA_n)}=\frac{\deg ( \partial \cA_n) }{\vol(\cA_n)}.
\]
We also need another subgraph $\wt{\cA}''$ of $\wt{\cA}$ obtained by removing the edges of $\wt{\cA}$ connecting $S_{n+1}$ with $S_{n}\setminus\partial\wt{\cA}$ and also $S_{n}\setminus\partial\wt{\cA}$ with the vertices in $S_{n-1}\cap \partial\wt{\cA}$. The obtained graph $\wt{\cA}''$ is a connected subgraph of $\cA_{n-1}$ and hence satisfies the induction hypothesis \eqref{eq:iso_n}. Our aim is to show that
\be\label{eq:iso_prime}
\frac{ \deg( \partial \wt{\cA})}{\vol (\wt{\cA})} \ge \min\Big\{\frac{ \deg( \partial \wt{\cA}')}{\vol (\wt{\cA}')},\frac{ \deg( \partial \wt{\cA}'')}{\vol (\wt{\cA}'')} \Big\},
\ee
Denoting $M := \#(S_n\cap \cV(\wt\cA))$ and $N:= \#(S_{n-1}\cap\partial\wt{\cA})$, we get
\be\label{eq:vol'}
\vol (\wt{\cA}') = \vol (\wt \cA) + (s_{n} - M) s_{n+1}  \ell_n  + (s_{n} - M)N \ell_{n-1},
\ee
and 
\be\label{eq:vol''}
\vol( \wt{\cA}'') = \vol( \wt{\cA}) - M s_{n+1} \ell_n - M N \ell_{n-1}.
\ee
Moreover, a careful inspection shows that 
\be\label{eq:deg'}
\deg (\partial \wt{\cA}') \le \deg(\partial \wt \cA) + (s_n - M) (s_{n+1}-s_{n-1} +  2N),
\ee
and
\be\label{eq:deg''}
\deg( \partial \wt \cA) = \deg( \partial \wt{\cA}'') + M (s_{n+1}-s_{n-1} + 2N).
\ee
Now observe that if \eqref{eq:iso_prime} fails to hold, then \eqref{eq:vol''} and \eqref{eq:deg''} would imply 
\be\label{eq:alphaest2} 
\frac{s_{n+1} + 2N - s_{n-1}}{s_{n+1} \ell_n + N \ell_{n-1}} < \frac{\deg(\partial \wt{\cA})}{\vol ( \wt{\cA})},
\ee
and, moroever, \eqref{eq:vol'} and \eqref{eq:deg'} lead to 
\be
\frac{s_{n+1} +  2N - s_{n-1}}{ s_{n+1}  \ell_n  +N\ell_{n-1}} > \frac{\deg(\partial \wt{\cA})}{\vol ( \wt{\cA})}.
\ee	 
This contradiction proves \eqref{eq:iso_prime} and hence finishes the proof of \eqref{eq:alpha}.
\end{proof}

\begin{remark}\label{rem:cheeger}
A few remarks are in order.
\begin{itemize}
\item[(i)] By the Cheeger-type estimate \cite[Theorem 3.4]{kn19}, we have
\be\label{eq:cheeger}
\lambda_0(\bH) \ge \frac{1}{4}\alpha(\cA)^2. 
\ee
Comparing \eqref{eq:cheeger} and \eqref{eq:alpha} with \eqref{eq:estKK01} and \eqref{eq:Iso02}, we conclude that positivity of the isoperimetric constant is indeed only sufficient for $\lambda_0(\bH)>0$. For example, $\alpha(\cA) = 0$ whenever $\vol(\cA) = \infty$ and $\{s_ns_{n+1}\}_{n\ge 0}$ has a bounded subsequence. 
\item[(ii)] The isoperimetric constant $\alpha(\cA)$ measures the ratio of the number of boundary points of $\cA_n$ to the volume of $\cA_n$ and thus provides a lower bound for $\lambda_0(\bH)$. The volume growth estimate \eqref{eq:volgr} provides an upper bound by relating the exponential growth of the volume of $\cA_n$ with respect to its diameter. Notice that the volume of the subgraph $\cA_n$ also appears in \eqref{eq:Iso01}--\eqref{eq:Iso02}. The meaning of the other quantity in \eqref{eq:Iso02}, namely, of $\sum_{k\ge n} \frac{\ell_k}{s_ks_{k+1}}$, which however provides two-sided estimates, remains unclear to us. 
\end{itemize}
\end{remark}

\section{Singular spectrum}\label{sec:SCspec}

Using the isometric isomorphism $U_\mu\colon f\mapsto \sqrt{\mu}f$ between Hilbert spaces $L^2(\cI_\cL;\mu)$ and $L^2(\cI_\cL)$, it is straightforward to check that the pre-minimal operator $\rH_0$ defined in Section \ref{ss:3.1} is unitarily equivalent to the operator $\wt{\rH}_0$ defined in $L^2(\cI_\cL)$ by
\begin{align}
\wt{\rH}_0 f &= -f'',\quad f\in \dom(\wt{\rH}_0) = U_\mu(\dom(\rH_0))\\
 \dom(\wt{\rH}_0) &= \Big\{f\in L^2_c([0,\cL))|\, \frac{1}{\sqrt{\mu}}f, \sqrt{\mu}f'\in AC_{\loc}(\cI_\cL),\ f'(0)=0,\ f''\in L^2(\cI_\cL)\Big\}. \nonumber
\end{align}
Since $\mu$ is piece-wise constant on $(0,\cL)$, the domain of $\wt{\rH}_0$ consists of compactly supported functions $f\in L^2_c(\cI_\cL)$ such that $f\in H^2(I_n)$ for all $n\ge 0$ and also satisfying the following boundary conditions
\[
f'(0)=0;\qquad f(t_n+) = \sqrt{\frac{s_{n+1}}{s_{n-1}}}f(t_n-),\quad f'(t_n+) = \sqrt{\frac{s_{n-1}}{s_{n+1}}}f'(t_n-),
\]
for all $n\ge 1$. Denote the closure of $\wt{\rH}_0$ by $\wt\rH$. The operator $\wt\rH$ has actively been studied  since its spectral properties play a crucial role in understanding spectral properties of Kirchhoff Laplacians on radial metric trees (let us only mention \cite{bf,ess}). It  turns out that one can immediately apply most of the results from \cite{bf} and \cite{ess} in order to prove the corresponding spectral properties of Kirchhoff Laplacians on radially symmetric antitrees. However, we need the following assumptions on the geometry of metric antitrees:

\begin{hypothesis}\label{hyp:remling}
There is a positive lower bound on the edge lengths, $\ell_\ast(\cA):=\inf_{n\ge 0} \ell_n >0$, and sphere numbers are such that 
\be\label{hyp:remling2}
\liminf_{n\ge 0}\frac{s_{n+2}}{s_n}>1. 
\ee
\end{hypothesis}

In this case clearly $\cL = \sum_{n\ge 0} \ell_n = \infty$ and hence both operators $\rH$ and $\wt{\rH}$ are self-adjoint. 
The next result is the analog of \cite[Theorem 2]{bf}.

\begin{theorem}\label{th:rem01}
Assume Hypothesis \ref{hyp:remling}. If in addition
\be\label{eq:sol}
\sup_{n\ge 0} \ell_n = \infty,
\ee
then $\sigma(\bH) = \R_{\ge 0}$ and $\sigma_{\ac}(\bH)=\emptyset$.
\end{theorem}

\begin{proof}
By Theorem \ref{th:decomp}, it suffices to show that $\sigma(\wt\rH) = \R_{\ge 0}$ and $\sigma_{\ac}(\wt\rH)=\emptyset$ since $\wt\rH = U_\mu\rH U^{-1}_\mu$. However, the latter follows from \cite[Theorem 6]{bf}.
\end{proof}

Moreover, using the results from \cite[\S 4]{koma10} and arguing as in the proof of \cite[Theorem 1]{mik96} (see also \cite[Theorem 5.20]{EKMN}), one can prove the following statement. 

\begin{theorem}\label{th:km01}
Assume Hypothesis \ref{hyp:remling}. If in addition
\be\label{eq:km01}
\sup_{n\ge 0} \frac{s_{n+2}}{s_n} = \infty,
\ee
then $\sigma_{\ac}(\bH)=\emptyset$.
\end{theorem}

In contrast to radially symmetric trees, antitrees always have a rather rich point spectrum (see Theorem \ref{th:decomp}). Moreover, under the assumptions of Hypothesis \ref{hyp:remling} this point spectrum is not a discrete subset, that is, it has finite accumulation points (see Remark \ref{rem:sigma_n}). On the other hand, similar to \cite[Theorem 7]{bf}, we can construct a class of antitrees such that $\sigma(\rH)$ is purely singular continuous. Moreover, it is possible to show that under the assumption $\ell_\ast(\cA)>0$ this situation is in a certain sense typical (cf. \cite[Theorems 4 and 8]{bf}).  Let us only mention the following Remling-type result (cf. \cite[Theorem 1.1]{rem}). 

\begin{theorem}\label{th:rem02}
Assume Hypothesis \ref{hyp:remling}. Also, assume that the sets $\{\ell_n\}_{n\ge 0}$ and $\{\frac{s_{n+2}}{s_n}\}_{n\ge 0}$ are finite. Then $\sigma_{\ac}(\bH)\neq \emptyset$ if and only if the sequence $\{(\ell_n,\frac{s_{n+2}}{s_n})\}_{n\ge 0}$ is eventually periodic.
\end{theorem} 

The proof is again omitted since it is analogous to that of \cite[Theorem 5.1]{ess}.

\section{Absolutely continuous spectrum}\label{sec:ACspec}

The decomposition \eqref{eq:Hdecomp} shows that 
\be
\sigma_{\ac}(\bH) = \sigma_{\ac}(\rH)
\ee
and both have multiplicity at most $1$. The results of the previous section show that antitrees with nonempty absolutely continuous spectrum is a rare event.  
Our main aim in this section is to apply two recent result from \cite{bd17} and \cite{ek18} on the absolutely continuous spectrum of Krein and generalized indefinite strings, respectively, in order to construct several classes of antitrees with rich absolutely continuous spectra, however, which are not eventually periodic in the sense of Theorem \ref{th:rem02}.  
We begin with the following result.

\begin{theorem}\label{th:ac01}
Let $\cA$ be an infinite radially symmetric antitree such that
\[
\cL = \sum_{n\ge 0}\ell_n = \infty.
\] 
Also, let $\mu$ be the function given by \eqref{eq:muL}. If 
\be\label{eq:ac01} 
\sum_{n\ge 0} \left(\int_n^{n+2} \mu(x)dx\int_n^{n+2}\frac{dx}{\mu(x)} - 4\right)<\infty,
\ee
then $\sigma_{\ac}(\bH) = \R_{\ge 0}$.
\end{theorem}

\begin{proof}
We only need to use Theorem 2 from \cite{bd17}. Indeed, as we know (see the proof of Theorem \ref{th:SpEst}), the operator $\rH$ is unitarily equivalent to the Krein string operator $\wt{\rH}$ given by \eqref{eq:wtH}--\eqref{eq:g_def}. 
Applying now Theorem 2 from \cite{bd17} to the operator $\wt\rH$, after straightforward calculations the corresponding condition (1.9) from \cite{bd17} turns into \eqref{eq:ac01}. 
\end{proof}

\begin{remark}\label{rem:bessden}
Let us mention that in Theorem \ref{th:ac01}, upon suitable modifications of \cite[Theorem 2]{bd17}, one can replace the intervals $(n,n+2)$ by intervals $\cI_n$, $n\ge 0$ which ``asymptotically" behave like $(n,n+2)$ (actually, by intervals with lengths uniformly bounded from above as well as by a positive constant from below  and satisfying a suitable overlapping property \cite{bess}), however, one has to replace 4 by a square of the length of the corresponding interval:
\be\label{eq:ac01mod} 
\sum_{n\ge 0} \left(\int_{\cI_n} \mu(x)dx\int_{\cI_n}\frac{dx}{\mu(x)} - |\cI_n|^2\right)<\infty.
\ee
\end{remark}

Let us first demonstrate the above result by considering an example of equilateral antitrees and then we shall extend it to a much wider setting (see Theorem \ref{cor:ac01} below). 

\begin{corollary}[Equilateral antitrees]\label{cor:ac_equilat}
Let $\cA$ be an infinite radially symmetric antitree with $\ell_n = \ell>0$  
for all $n\ge0$. If 
\be\label{eq:remL1}
\sum_{n\ge 0} \Big(\frac{s_{n+2}}{s_n} - 1\Big)^2 <\infty,
\ee
then $\sigma_{\ac}(\bH) = \R_{\ge 0}$.
\end{corollary}

\begin{proof}
Setting $\cI_n = (\ell n,\ell(n+2))$, $n\ge 0$, straightforward calculations show that
\begin{align*}
\int_{\cI_n} \mu(x)dx & \int_{\cI_n}\frac{dx}{\mu(x)} - |\cI_n|^2 \\
 & = (s_ns_{n+1} + s_{n+1}s_{n+2})\Big(\frac{1}{s_ns_{n+1}} + \frac{1}{s_{n+1}s_{n+2}}\Big)\ell^2 -4\ell^2 \\
& = \frac{(s_{n+2} + s_n)^2}{s_ns_{n+2}}\ell^2 - 4\ell^2 = \ell^2\frac{(s_{n+2} - s_n)^2}{s_ns_{n+2}} = \ell^2\frac{s_n}{s_{n+2}}\Big(\frac{s_{n+2}}{s_n} - 1\Big)^2. 
\end{align*}
Theorem \ref{th:ac01} and Remark \ref{rem:bessden} complete the proof.
\end{proof}

\begin{remark}\label{rem:8.4}
First of all, Corollary \ref{cor:ac_equilat} demonstrates that \eqref{hyp:remling2} is essential for the results of Section \ref{sec:SCspec}. Let us also mention that it is possible to show by using the results of \cite[\S 4.2]{koma10} that the stronger condition
\be\label{eq:remL2}
\sum_{n\ge 0} \Big|\frac{s_{n+2}}{s_n} -1\Big| <\infty
\ee
holds exactly when the operator $\wt{\rH}$ considered in Section \ref{sec:SCspec} is a trace class perturbation (in the resolvent sense) of the free Hamiltonian $-\frac{d^2}{dx^2}$ acting in $L^2(\R_+)$ and hence in this case the Birman--Krein theorem implies $\sigma_{\ac}(\rH) = \R_{\ge 0}$. However, \eqref{eq:remL2} does not hold already for polynomially growing equilateral antitrees, e.g., take $s_n = n+1$ (see also Section \ref{ss:polAT}). Moreover, \eqref{eq:remL1} is equivalent to the fact that $\wt{\rH}$ is a Hilbert--Schmidt class perturbation (in the resolvent sense) of the free Hamiltonian.
\end{remark}

The rather strong assumption that $\cA$ is equilateral can indeed be replaced by $\ell_\ast(\cA) >0$.  
In order to do this, it will turn out useful to rewrite \eqref{eq:ac01}. 
Let 
\be\label{eq:Mran}
\cM : = \ran(\mu) = \{ s_n s_{n+1}\colon \; n \in \Z_{\ge 0} \}
\ee
be the image of the function $\mu$ defined in \eqref{eq:muL}. For every $s \in \cM$, we set 
\be\label{eq:}
\cI_s := \mu^{-1}(\{ s \}) = \{ x \in [0, \infty)\colon \; \mu(x) = s \},
\ee
that is, $\cI_s$ is the preimage of $\{s\} \in \cM$ with respect to $\mu$. 

\begin{lemma} \label{lem:ac01}
	Let $\cA$ be an infinite radially symmetric antitree with $\cL=\infty$. Then 
	\be \label{eq:lemac01}
		\sum_{n\ge 0} \left(\int_n^{n+2} \mu(x)dx\int_n^{n+2}\frac{dx}{\mu(x)} - 4\right) = \frac{1}{2}  \sum_{n\ge 0} \sum_{s \in \cM} \sum_{\xi \neq s}  |\cI_s^n||\cI_\xi^n| \frac{(s - \xi)^2}{s \xi} ,
	\ee
	where $|\cI_s^n|$ is the Lebesgue measure of $\cI_s^n := \cI_s \cap (n, n+2)$.
\end{lemma}

\begin{proof}

For every fixed $n \in \Z_{\ge 0}$, we clearly have
	\begin{align*}
		\int_n^{n+2} &\mu(x)dx\int_n^{n+2}\frac{dx}{\mu(x)} = 
		\Big ( \sum_{s \in \cM} s | \cI_s^n| \Big)  \Big( \sum_{\xi \in \cM} \frac{1}{\xi}| \cI_\xi^n| \Big) \\
	=& \sum_{s  \in \cM } \sum_{\xi \neq s} | \cI_s^n|  | \cI_\xi^n| \frac{s}{\xi} + \sum_{s \in \cM} | \cI_s^n|^2 
	= \frac{1}{2} \sum_{s  \in \cM } \sum_{\xi \neq s}  | \cI_s^n|  | \cI_\xi^n| \left  (\frac{\xi}{s} + \frac{s}{\xi} \right) + \sum_{s \in \cM}| \cI_s^n|^2.
\end{align*}
Moreover, by construction 
\be\label{eq:sumIn}
\sum_{s \in \cM} | \cI_s^n| =  2,
\ee
 and hence
\begin{align*}
	\sum_{s \in \cM} | \cI_s^n|^2 - 4 = \sum_{s \in \cM} | \cI_s^n|  (| \cI_s^n| - 2) = 
	- \sum_{s \in \cM} \sum_{\xi \neq s}  | \cI_s^n|  | \cI_\xi^n|.
\end{align*}
	Combining the last two equalities, we get
	\begin{align*}
	\int_n^{n+2} \mu(x)dx \int_n^{n+2}\frac{dx}{\mu(x)} - 4 &
	= \frac{1}{2} \sum_{s \in \cM} \sum_{\xi \neq s}  | \cI_s^n|  | \cI_\xi^n| \left  (\frac{\xi}{s} + \frac{s}{\xi} -2 \right ) \\
		&= \frac{1}{2}  \sum_{s \in \cM} \sum_{\xi \neq s} | \cI_s^n|  | \cI_\xi^n| \frac{(s - \xi)^2}{s \xi},
	\end{align*}
	which completes the proof.
\end{proof}
 
\begin{theorem} \label{cor:ac01}
Let $\cA$ be an infinite radially symmetric antitree with sphere numbers satisfying \eqref{eq:remL1}. If
\[ 
\ell_\ast( \cA) = \inf_{n\ge 0}\ell_n > 0,
\]
then $\sigma_{\ac}(\bH) = \R_{\ge 0}$.
\end{theorem}

\begin{proof}
Suppose $\ell_\ast (\cA) \ge 2$. Then, by Lemma \ref{lem:ac01}, for every $n\in\Z_{\ge0}$, we get
\[
\int_n^{n+2} \mu(x)dx \int_n^{n+2}\frac{dx}{\mu(x)} - 4 = \frac{1}{2}  \sum_{s \in \cM} \sum_{\xi \neq s}  |\cI_s^n||\cI_\xi^n| \frac{(s - \xi)^2}{s \xi} \le  \sum_{s \in \cM_n} \sum_{\xi \neq s} |\cI_\xi^n|\frac{(s - \xi)^2}{s \xi},
\]
where $\mathcal M_n :=  \mu\big((n,n+2)\big) = \{s_k s_{k+1} \colon (n, n+2) \cap I_k \neq \varnothing \}$. Since $\ell_k \ge 2$ for all $k\ge0$ by assumption, $\mu$ is either constant on $(n, n+2)$ or attains precisely two different values. In the first case, the righthand side is equal to zero. In the second, we obviously get the estimate
	\[
	\int_n^{n+2} \mu(x)dx \int_n^{n+2}\frac{dx}{\mu(x)} - 4\le 	
	2\sum_{t_k \in (n, n+2)}\frac{( s_{k+1} - s_{k-1})^2   }{s_{k-1} s_{k+1}}.
	\]
Thus we end up with the  following bound
\begin{align*}
\sum_{n\ge 0} \left(\int_n^{n+2} \mu(x)dx \int_n^{n+2}\frac{dx}{\mu(x)} - 4\right) & \le 
      2  \sum_{n\ge 0}\sum_{t_k \in (n, n+2)}\frac{( s_{k+1} - s_{k-1})^2   }{s_{k-1} s_{k+1}} \\ 
      & \le  4  \sum_{n\ge 0}\frac{( s_{n+2} - s_{n})^2   }{s_{n} s_{n+2}} <\infty,
\end{align*}
which proves the claim by applying Theorem \ref{th:ac01}. 

It remains to note that the general case $\ell_\ast (\cA) >0$ can be reduced to the one with $\ell_\ast (\cA) \ge 2$ by using the standard scaling argument (see also Remark \ref{rem:bessden}). 
	\end{proof}

In fact, one can extend the above result to the case when lengths do not admit a strictly positive lower bound. However, in this case one has to modify \eqref{eq:remL1} in an appropriate way.

\begin{lemma}\label{lem:ac05}
Let $\cA$ be an infinite radially symmetric antitree with $\cL=\infty$. 
Also, let $\ell_n \le 1$ for all $n\ge 0$ and $\ell_n = o(1)$ as $n\to \infty$. If $\{s_n\}_{n\ge 0}$ is a nondecreasing sequence such that
 \be\label{eq:remL!mod}
 \sum_{n\ge 0} \left(\frac{s_{m(n+2)}}{s_{m(n)}}-1 \right)^2 <\infty,
 \ee
 then $\sigma_{\ac}(\bH) = \R_{\ge 0}$. 
 
 Here for each $n\in\Z_{\ge0}$ the natural number $m(n)$ is defined by
 \be\label{eq:def_mn}
t_{m(n)} \le n < t_{m(n)+1}, \qquad t_n = \sum_{k=0}^{n-1} \ell_k. 
\ee
 \end{lemma}
 
 \begin{proof}
 Set $\cI_n:= (t_{m(n)},t_{m(n+2)+1})$, $n\ge 0$. By construction $(n,n+2)\subseteq \cI_n$ for all $n\ge 0$ and $|\cI_n\setminus (n,n+2)| =o(1)$ as $n\to \infty$. 
Thus, by Theorem \ref{th:ac01} and Remark \ref{rem:bessden}, it suffices to show that 
\be
\sum_{n\ge 0} \underbrace{\left(\int_{t_{m(n)}}^{t_{m(n+2)+1}} \mu(x)dx\int_{t_{m(n)}}^{t_{m(n+2)+1}} \frac{dx}{\mu(x)} - (t_{m(n+2)+1} - t_{m(n)})^2\right)}_{=:R_n} <\infty.
\ee 
Since $\mu$ is given by \eqref{eq:muL}, we get
 \begin{align*}
 R_n=\sum_{k=m(n)}^{m(n+2)} &s_ks_{k+1}\ell_k \sum_{k=m(n)}^{m(n+2)} \frac{\ell_k}{s_ks_{k+1}} - \Big(\sum_{k=m(n)}^{m(n+2)} \ell_k\Big)^2  
  = \sum_{k,j=m(n)}^{m(n+2)} \ell_k\ell_j \Big(\frac{s_js_{j+1}}{s_ks_{k+1}}-1\Big) \\
& =2 \sum_{m(n)\le k < j \le m(n+2)} \ell_k\ell_j \frac{(s_js_{j+1} - s_ks_{k+1})^2}{s_ks_{k+1}s_js_{j+1}} \\
& \le 2  \sum_{m(n)\le k < j \le m(n+2)} \ell_k\ell_j \frac{(s_{m(n+2)+1}^2 - s_{m(n)}^2)^2}{s_{m(n)}^4} \\
& \lesssim 2 \sup_{k\ge 0}|\cI_k|^2 \left(\frac{s_{m(n+2)}^2}{s_{m(n)}^2}-1 \right)^2 \lesssim  \left(\frac{s_{m(n+2)}}{s_{m(n)}}-1 \right)^2
\end{align*}
for all $n\ge 0$ if $\frac{s_{m(n+2)}}{s_{m(n)}} = 1+o(1)$. 
 \end{proof}

\begin{remark}
In fact, the assumptions on lengths that $\ell_n \le 1$ for all $n\ge 0$ and $\ell_n = o(1)$ as $n\to \infty$ as well as monotonicity of sphere numbers are superfluous and we need them for simplicity only. Of course, one can considerably weaken them, however, the analysis becomes more involved and cumbersome.
\end{remark}

We finish this section with another result based on \cite{ek18}, which also allows to construct antitrees with absolutely continuous spectrum supported on $\R_{\ge 0}$.

\begin{theorem}\label{th:ac02}
Let $\cA$ be an infinite radially symmetric antitree such that $\vol(\cA)=\infty$ and $\cL_\mu=\infty$.
If there are constants $a\in \R$ and $b\in \R_{> 0}$ such that
\be\label{eq:ac02}
\int_0^\cL \frac{1}{\mu(x)}\Big| \int_0^{x} \Big(\mu(s) - \frac{b}{\mu(s)}\Big)ds - a \Big|^2 dx<\infty,
\ee
where $\mu$ is given by \eqref{eq:muL}, then $\sigma_{\ac}(\bH) = \R_{\ge 0}$.
\end{theorem}

\begin{proof}
As in the proof of Theorem \ref{th:ac01}, we know that the operator $\rH$ is unitarily equivalent to the operator $\wt\rH$.
By Theorem 3.1 from \cite{ek18}, $\sigma_{\ac}(\wt\rH) = [0,\infty)$ if there are constants $a\in \R$ and $b\in \R_{> 0}$ such that
\[
\int_0^\infty \left| M(x) - a - b x\right|^2 dx  < \infty,
\]
where $M$ is defined by \eqref{eq:defM}. 
Straightforward calculations finish the proof.
\end{proof}

\begin{remark}
For a string operator defined by \eqref{eq:wtH}, Theorem \ref{th:ac01} and Theorem \ref{th:ac02} also imply that the entropy, respectively, some sort of relative entropy of the corresponding spectral measure is finite (see \cite{bd17} for details). However, the meaning of this fact for the corresponding quantum graph operator $\bH$ is unclear to us.
\end{remark}

\section{Examples}\label{sec:Example}

\subsection{Exponentially growing antitrees}\label{ss:expAT}
Consider an exponentially growing antitree from \cite[Example 8.6]{kn19}. Namely, fix $\beta \in \Z_{\ge 2}$ and let $\cA_\beta$ be the antitree with sphere numbers $s_n = \beta^n$, $n\ge 0$. Suppose that $\{\ell_n\}_{n\ge 0}$ are the lengths. Notice that
\be\label{eq:ex1sa}
\vol(\cA_\beta) = \sum_{n\ge 0} \beta^{2n+1}\ell_n.
\ee
Then the basic spectral properties of the corresponding quantum graph operator are contained in the following proposition.

\begin{proposition}\label{prop:Hbeta}
Let $\bH^\beta$ be the quantum graph operator associated with the antitree $\cA_\beta$. Then:
\begin{itemize}
\item[(i)] The operator $\bH^\beta$ is self-adjoint if and only if the series in \eqref{eq:ex1sa} diverges. 
\item[(ii)] If $\vol(\cA_\beta)<\infty$, then deficiency indices of $\bH^\beta$ are equal to $1$. Moreover,  the spectra of self-adjoint extensions of $\bH^\beta$ are purely discrete and  eigenvalues admit the standard Weyl asymptotic \eqref{eq:WeylLaw}. 
\end{itemize}
Assume in addition that $\vol(\cA_\beta)=\infty$.
\begin{itemize}
\item[(iii)] The spectrum of $\bH^\beta$ is purely discrete if and only if $\ell_n = o(1)$ as $n\to \infty$. 
\item[(iv)] The resolvent of $\bH^\beta$ belongs to the trace class if and only if 
\begin{align}\label{eq:trace_beta}
\sum_{n\ge 0}\beta^{2n}\ell_n^2 < \infty.
\end{align}
\item[(v)] $\bH^\beta$ is positive definite if and only if $\ell^\ast(\cA_\beta) <\infty$. Moreover, in this case
\begin{align}
\frac{1}{4C} \le & \lambda_0(\bH^\beta) \le \frac{1}{C}, & \frac{1}{4C_{\ess}} \le & \lambda_0^{\ess}(\bH^\beta) \le \frac{1}{C_{\ess}},
\end{align}
where 
\be\label{eq:Cbeta}
\sup_{n\ge 0}\sum_{k=0}^n \beta^{2k}\ell_k \sum_{k\ge n+1} \frac{\ell_k}{\beta^{2k}} \le C \le \sup_{n\ge 0} \sum_{k=0}^n \beta^{2k}\ell_k \sum_{k\ge n} \frac{\ell_k}{\beta^{2k}},
\ee
and
\be\label{eq:Cessbeta}
\lim_{m\to \infty} \sup_{n\ge m}\sum_{k=m}^n \beta^{2k}\ell_k \sum_{k\ge n+1} \frac{\ell_k}{\beta^{2k}} \le C_{\ess} \le \lim_{m\to \infty} \sup_{n\ge m}\sum_{k=m}^n \beta^{2k}\ell_k \sum_{k\ge n} \frac{\ell_k}{\beta^{2k}}.
\ee
\end{itemize}
\end{proposition}

\begin{proof}
Items (i) and (ii) follow from Theorem \ref{th:SA} and Corollary \ref{cor:finvol}. 

(iii) Applying Theorem \ref{th:discr} (see also Remark \ref{rem:discr}), we only need to show that $\ell_n = o(1)$ as $n\to \infty$ is sufficient for the discreteness. 
Indeed, we can estimate 
\begin{align}\label{eq:estimate_beta}
\begin{split}
\sum_{k=0}^n &\beta^{2k}\ell_k \sum_{k\ge n} \frac{\ell_k}{\beta^{2k}} 
 \le \ell^\ast(\cA_\beta)\sup_{k\ge n}\ell_k \sum_{k=0}^n \beta^{2k} \sum_{k\ge n} \frac{1}{\beta^{2k}} \\
& = \ell^\ast(\cA_\beta)\sup_{k\ge n}\ell_k \frac{\beta^{2n+2} - 1}{\beta^{2n+2}} \Big(\frac{\beta^2}{\beta^2-1}\Big)^2
< \frac{\ell^\ast(\cA_\beta)}{(1-\beta^{-2})^2}\,\sup_{k\ge n}\ell_k,
\end{split}
\end{align}
where $ \ell^\ast(\cA_\beta) = \sup_{n\ge0}\ell_n$. 
Hence \eqref{eq:KK07} is satisfied if $\ell_n= o(1)$. 

(iv) Clearly, \eqref{eq:trace_beta} coincides with condition (i) of Theorem \ref{th:trace} and hence it is necessary. 
Applying the Cauchy--Schwarz inequality, we get the following estimate:
\begin{align*}
	 \sum_{n\ge 0}\frac{\ell_n}{s_ns_{n+1}} & \sum_{k=0}^{n-1}s_ks_{k+1}\ell_k  = \sum_{n \ge 0} \frac{\ell_n}{\beta^{2n}} \sum_{k=0}^{n-1} \beta^{2k} \ell_k 	\\
	 & \le \sum_{n \ge 0} \frac{\ell_n}{\beta^{2n}} \Big ( \sum_{k = 0}^{n-1} \beta^{2k} \ell_k^2   \sum_{k = 0}^{n-1} \beta^{2k} \Big)^{1/2}   
	  = \sum_{n \ge 0} \frac{\ell_n}{\beta^{2n}} \Big( \frac{\beta^{2n}-1}{\beta^2-1}\sum_{k = 0}^{n-1} \beta^{2k} \ell_k^2\Big)^{1/2}   \\
	& <  \sum_{n \ge 0} \frac{\ell_n}{\beta^{n}} \Big(\sum_{k = 0}^{n-1} \beta^{2k} \ell_k^2\Big)^{1/2} 
	< \frac{\ell^\ast(\cA_\beta)}{\beta-1} \Big(\sum_{k \ge 0} \beta^{2k} \ell_k^2\Big)^{1/2}.
\end{align*}
Therefore, \eqref{eq:trace_beta} implies condition (ii) of Theorem \ref{th:trace}, which proves the claim.

(v) immediately follows from \eqref{eq:estimate_beta}, Theorems \ref{th:SpEst} and \ref{th:SpEstEss} and Remark \ref{rem:SpEst01}.
\end{proof}

\begin{remark}
\begin{itemize}
\item[(i)] 
Both the discreteness and uniform positivity criteria for $\bH^\beta$ were obtained in \cite[Example 8.6]{kn19}. Notice that these results are a consequence of the positivity of the combinatorial isoperimetric constant in this case (see \cite{kn19}).
Moreover, using the rough estimate \eqref{eq:estimate_beta}, one would be able to recover the lower bounds (8.9) and (8.10) from \cite{kn19}.
\item[(ii)]  
It is impossible to apply Theorem \ref{th:ac01} and Theorem \ref{th:ac02} to $\cA_\beta$ (this either can be seen from Proposition \ref{prop:Hbeta}(v) or one can prove that both conditions \eqref{eq:ac01} and \eqref{eq:ac02} are always violated if sphere numbers grow exponentially). 
\item[(iii)] 
Since the sphere numbers of $\cA_\beta$ satisfy
\[
\frac{s_{n+2}}{s_{n}} = \beta^2
\]
for all $n\ge 0$, we can apply the results of Section \ref{sec:SCspec}. Namely, under the additional assumption $\ell_\ast(\cA_\beta)>0$, we conclude that the absolutely continuous spectrum of $\bH$ is in general empty. In particular, it is always the case if $\ell^\ast(\cA_\beta) = \infty$ (Theorem \ref{th:rem01}). Moreover, assuming that $\{\ell_n\}_{n\ge 0}$ is a finite set, by Theorem \ref{th:rem02}, $\sigma_{\ac}(\bH)\neq \emptyset$ would imply that the sequence $\{\ell_n\}_{n\ge 0}$ is eventually periodic.  
\item[(iv)] 
Notice that the isoperimetric constant is given by (see \eqref{eq:alpha}) 
\[
\frac{1}{\alpha(\cA_\beta)} 
= \sup_{n\ge 0}\frac{1}{\beta^{2n}}{\sum_{k=0}^n \beta^{2k}\ell_k}.
\]
\end{itemize}
\end{remark}

\subsection{Polynomially growing antitrees} \label{ss:polAT}
Fix $q \in \Z_{\ge 1}$ and let $\cA^q$ be the antitree with sphere numbers $s_n = (n+1)^q$, $n\ge 0$ (the case $q=1$ is depicted in Figure \ref{fig:antitree}). Suppose that $\{\ell_n\}_{n\ge 0}$ are the lengths. Notice that
\be\label{eq:ex2sa}
\vol(\cA^q) = \sum_{n\ge 0} (n+1)^q(n+2)^q \ell_n.
\ee
Then the basic spectral properties of the corresponding quantum graph operator are contained in the following proposition.

\begin{proposition}\label{prop:Hq}
Let $\bH^q$ be the quantum graph operator associated with the antitree $\cA^q$. Then:
\begin{itemize}
\item[(i)] The operator $\bH^q$ is self-adjoint if and only if 
\be\label{eq:finvol_q}
\sum_{n\ge 0} n^{2q}\ell_n = \infty.
\ee
\item[(ii)] If the series in \eqref{eq:finvol_q} converges, then deficiency indices of $\bH^q$ are equal to $1$. Moreover,  the spectra of self-adjoint extensions of $\bH^q$ are purely discrete and  eigenvalues admit the standard Weyl asymptotic \eqref{eq:WeylLaw}. 
\end{itemize}
Assume in addition that \eqref{eq:finvol_q} is satisfied, that is, $\bH^q$ is self-adjoint.
\begin{itemize}
\item[(iii)] The spectrum of $\bH^q$ is purely discrete if and only if 
\be\label{eq:Hq_discr}
\lim_{n\to \infty}\sum_{k=0}^n k^{2q}\ell_k \sum_{k\ge n} \frac{\ell_k}{k^{2q}} = 0.
\ee
In particular, the spectrum is purely discrete if $\ell_n = o(n^{-1})$ as $n\to \infty$.
\item[(iv)] The resolvent of $\bH^q$ belongs to the trace class if and only if 
\begin{align}\label{eq:trace_q}
\sum_{n\ge 0}n^{2q}\ell_n^2 < \infty.
\end{align}
\item[(v)] $\bH^q$ is positive definite if and only if
\be\label{eq:Hq_positive}
\sup_{n\ge 1}\sum_{k=0}^n k^{2q}\ell_k \sum_{k\ge n} \frac{\ell_k}{k^{2q}} <\infty.
\ee
In particular, $\lambda_0(\bH^q)>0$ if $\ell_n = \OO(n^{-1})$ as $n\to \infty$. 
\item[(vi)] If $\ell_\ast(\cA^q) >0$, then $\sigma_{\ac}(\bH^q) = \R_{\ge 0}$.
\end{itemize}
\end{proposition}

\begin{proof}
(i) and (ii) follow immediately from Theorem \ref{th:SA} and Corollary \ref{cor:finvol} since $\vol(\cA^q) = \infty$ exactly when \eqref{eq:finvol_q} is satisfied.

(iii) Applying Theorem \ref{th:discr} (see also Remark \ref{rem:discr}), we conclude that in the case \eqref{eq:finvol_q}, the operator $\bH$ has purely discrete spectrum if and only if 
\[
\lim_{n\to \infty}\sum_{k=0}^n (k^2+3k+2)^q\ell_k \sum_{k\ge n} \frac{\ell_k}{(k^2+3k+2)^q} = 0.
\]
It is not difficult to show that the latter is equivalent to \eqref{eq:Hq_discr}. Moreover, \eqref{eq:Hq_discr} holds true whenever $\ell_n = o(n^{-1})$ as $n\to \infty$ since 
\begin{align*}
\sum_{k=0}^n k^{2q-1} & = \frac{n^{2q}}{2q}(1+o(1)), & \sum_{k\ge n} \frac{1}{k^{2q+1}} & = \frac{n^{- 2q}}{2q}(1+o(1)).
\end{align*}

(iv) First observe that \eqref{eq:trace01} is equivalent to \eqref{eq:trace_q}. Moreover,  \eqref{eq:trace_q} implies also \eqref{eq:KK08}. Indeed, we get
\begin{align*}
\sum_{n\ge 0} &\frac{\ell_n}{(n^2+3n+2)^q} \sum_{k=0}^{n-1}(k^2+3k+2)^q\ell_k  < \sum_{n\ge 0} \frac{\ell_n}{(n+1)^{2q}}  \sum_{k=0}^{n-1}(k+2)^{2q} \ell_k \\
	 & \le  \sum_{n\ge 0} \frac{\ell_n}{(n+1)^{2q}} \Big(  \sum_{k=0}^{n-1} (k+2)^{2q} \ell_k^2  \,  \sum_{k=0}^{n-1} (k+2)^{2q}  \Big)^{1/2} \\
	 &\lesssim  \sum_{n\ge 0} \frac{\ell_n}{(n+1)^{2q}} \Big(  (n+1)^{2q+1}\sum_{k=0}^{n-1} (k+2)^{2q} \ell_k^2 \Big)^{1/2} \\
	&  <  \Big(  \sum_{k\ge 0} (k+2)^{2q} \ell_k^2 \Big)^{1/2}  \sum_{n\ge 0} \frac{\ell_n}{ (n+1)^{q- 1/2}} <   \sum_{k\ge 0} (k+2)^{2q} \ell_k^2  \Big(\sum_{n \ge 1} \frac{1}{n^{4q-1}} \Big)^{1/2},
\end{align*}
where the second and the last inequalities we obtained by applying the Cauchy--Schwarz inequality. 
It remains to use Theorem \ref{th:trace}.

(v) follows by applying Theorem \ref{th:SpEst} (see also Remark \ref{rem:SpEst01}).

(vi) Since
\[
\sum_{n\ge 0}\Big(\frac{s_{n+2}}{s_n}-1\Big)^2  = \sum_{n\ge 1}\Big(\frac{(n+2)^q}{n^q} - 1\Big)^2 \lesssim \sum_{n\ge 1}\frac{1}{n^2} = \frac{\pi^2}{6},
\]
the claim is immediate from Theorem \ref{cor:ac01}. 
\end{proof}

\begin{remark}\label{rem:infdi}
A few remarks are in order.
\begin{enumerate}
\item[(i)]  The antitree $\cA^q$ and the corresponding Kirchhoff Laplacian $\bH$ have been considered in \cite[Example 8.7]{kn19}. The analysis of spectral properties (in particular, spectral estimates) is a rather delicate task in this case since the combinatorial isoperimetric constant of $\cA^q$ is equal to $0$. 
We were able to describe basic spectral properties of $\bH^q$ only due to the presence of radial symmetry. Spectral properties of Kirchhoff Laplacians without radial symmetry seems to be a rather complicated problem -- even the self-adjointness problem (modulo some recent criteria obtained in \cite{EKMN}) is unclear to us at the moment. 
\iflong{}Let us only mention that there are examples of metric antitrees having finite volume and such that the corresponding Kirchhoff Laplacian has {\em infinite deficiency indices} \cite{kmn19}. \fi
\item[(ii)] It can be demonstrated by examples that the conditions $\ell_n = o(n^{-1})$ (resp., $\ell_n = \OO(n^{-1})$) as $n\to \infty$ are not necessary for the discreteness (resp., positivity). However, they are in a certain sense sharp (see \cite[Lemma 8.9]{kn19} and also Example \ref{ex:02} below). 
\item[(iii)] Since $s_{n+2} = s_n(1+o(1))$, 
we can't apply the results of Section \ref{sec:SCspec} (see Hypothesis \ref{hyp:remling}). Moreover, Proposition \ref{prop:Hq}(vi) shows that in general $\bH^q$ has absolutely continuous spectrum supported on $\R_{\ge 0}$. However, Theorem \ref{th:ac01} is a consequence of \cite[Theorem 2]{bd17}, which allows a presence of a rather rich singular (continuous) spectrum.
  \end{enumerate}
\end{remark}

 We can also improve Proposition \ref{prop:Hq}(vi) by allowing arbitrarily small lengths.
 
 \begin{corollary}\label{cor:ac04}
 Suppose $\ell_n \le 1$ for all $n\ge 0$ and $\ell_n = o(1)$ as $n\to \infty$. If 
 \be
 \sum_{n\ge 0} \left(\frac{{m(n+2)}}{{m(n)}}-1 \right)^2 <\infty,
 \ee
 then $\sigma_{\ac}(\bH^q) = \R_{\ge 0}$. 
 Here $m(n)$ is defined as in Lemma \ref{lem:ac05}.
 \end{corollary}
 
 \begin{proof}
We need to apply Lemma \ref{lem:ac05} and notice that in this case
\begin{align*}
\frac{s_{m(n+2)}}{s_{m(n)}}-1 = \left(\frac{m(n+2)+1}{m(n)+1}\right)^{q} - 1 \approx \frac{m(n+2)}{m(n)} - 1,  
\end{align*}
as $n\to \infty$.
 \end{proof}

\begin{example}\label{ex:02}
Fix $s\ge 0$. Let the lengths of the metric antitree $\cA^q$ be given by
\be
\ell_n = \frac{1}{(n+1)^s},\qquad n\ge 0.
\ee
Denote the corresponding Kirchhoff Laplacian by $\bH^{q,s}$. 
Applying Proposition \ref{prop:Hq} and Corollary \ref{cor:ac04}, we end up with the following description of the spectral properties of $\bH^{q,s}$.

\begin{corollary}\label{cor:Hqs}
\begin{itemize}
\item[(i)] $\bH^{q,s}$ is self-adjoint if and only if $s\in [0,2q+1]$.
If  $s>2q+1$, then then deficiency indices of $\bH^{q,s}$ are equal to $1$. Moreover, in this case the spectra of self-adjoint extensions $\bH^{q,s}_\theta$ of $\bH^{q,s}$ are purely discrete and  eigenvalues admit the standard Weyl asymptotic 
\be\label{eq:Weyl_qs}
\lim_{\lambda\to \infty} \frac{N(\lambda;\bH^{q,s}_\theta)}{\sqrt{\lambda}} = \frac{1}{\pi} \sum_{k=0}^q \binom{q}{k}\zeta(s-2q+k),
\ee 
where 
$\zeta$ is the Riemann zeta function.
\end{itemize}
Assume in addition that $s\in[0,2q+1]$, that is, $\bH^q$ is self-adjoint.
\begin{itemize}
\item[(ii)] The spectrum of $\bH^{q,s}$ is purely discrete if and only if $s\in (1,2q+1]$. Moreover, 
the resolvent of $\bH^{q,s}$ belongs to the trace class if and only if $s\in (q+1/2,2q+1]$.
\item[(iii)] $\bH^{q,s}$ is positive definite if and only if $s\in [1,2q+1]$. 
\item[(iv)] If $s\in[0,1)$, then $\sigma_{\ac}(\bH^{q,s}) = \R_{\ge 0}$.
\end{itemize}
\end{corollary}

We leave its proof to the reader and finish this section with a few  remarks.
\begin{remark}
Corollary \ref{cor:Hqs} complements the results obtained in \cite[Example 8.7]{kn19}.
Moreover, items (ii) and (iii) demonstrate sharpness of sufficient conditions obtained in Proposition \ref{prop:Hq}(iii) and (v). Let us only mention that the question on the structure of the essential spectrum of $\bH^{q,1}$ as well as on the structure of the singular spectrum of $\bH^{q,s}$ with $s\in [0,1]$ remains open.\hfill $\lozenge$
\end{remark}
\end{example}

\begin{remark}
In conclusion let us mention that choosing slightly different lengths
\[
\ell_n = \frac{(n+1)^{q-s}}{(n+2)^{q}},\quad n\ge 0,
\]
and denoting the corresponding operator by $\wt{\bH}^{q,s}$, we obtain
\be
\lim_{\lambda\to \infty}\frac{N(\lambda;\wt{\bH}^{q,s}_\theta)}{\sqrt{\lambda}} = \frac{1}{\pi}{\zeta(s-2q)}, \quad s>2q+1. 
\ee
\end{remark}

\bigskip
\noindent
\ack We thank the referee for the careful reading of our manuscript and critical remarks that have helped to improve the exposition.

\iflong

\appendix

\section{The spectrum of the operator $\wt{\rh}_n$}\label{app:atan2}

For every $n\ge 1$ consider the function $d_n\colon \R_{> 0}\to \R$ given by
\be\label{eq:d_n}
d_n(\lambda) = s_{n+1}\cos(\sqrt{\lambda}\ell_n)\sin(\sqrt{\lambda}\ell_{n-1}) + s_{n-1}\cos(\sqrt{\lambda}\ell_{n-1})\sin(\sqrt{\lambda}\ell_{n}).
\ee
It is straightforward to establish the following connection between $d_n$ and the spectrum of $\wt{\rh}_n$.

\begin{lemma}\label{lem:spec_tihn}
The spectrum of the operator $\wt{\rh}_n$ coincides with the set of positive zeros of the function $d_n$.
\end{lemma}

Notice that  
\be
d_n(\lambda) = \frac{s_{n+1}+s_{n-1}}{2}\sin\big(2\sqrt{\lambda}\ell_n\big),
\ee
if $\ell_{n-1} = \ell_{n}$ and hence $\sigma(\wt{\rh}_n) = \{\frac{\pi^2k^2}{(2\ell_n)^2}\}_{k\ge 1}$. Moreover, 
if $s_{n-1} = s_{n+1}$, then
\be
d_{n}({\lambda}) = s_{n+1}\sin\big(\sqrt{\lambda}(\ell_{n-1}+\ell_n)\big),
\ee
and hence $\sigma(\wt{\rh}_n) = \{\frac{\pi^2k^2}{(\ell_{n-1}+\ell_n)^2}\}_{k\ge 1}$. Notice that in both cases the lowest eigenvalue $\lambda_1 ({\wt \rh}_n)$ of $\wt{\rh}_n$ is 
\be
\lambda_1 ({\wt \rh}_n) = \frac{\pi^2}{(\ell_{n-1} +\ell_n)^2}.
\ee
In general, such an equality is not true. Indeed, except these two special cases, it is usually difficult to find a closed form for the eigenvalues of the operator $\wt{\rh}_n$. 
Our next aim is to estimate the lowest eigenvalue $\lambda_1(\wt{\rh}_n)$ of $\wt{\rh}_n$. We begin with the following observation.

\begin{corollary}
If $\ell_n^\ast := \max \{\ell_{n-1}, \ell_n \}$, then
\be \label{est:l1}
	\left(  \frac{\pi }{2  \ell_n^\ast } \right )^2 \le \lambda_1 ({\wt \rh}_n)  \le \left(  \frac{\pi }{ \ell_n^\ast } \right )^2,
\ee
\end{corollary}

\begin{proof}
First of all, without loss of generality we can assume that $\ell_n^\ast = \ell_n$. If $0<\sqrt{\lambda}<\pi/ 2 \ell_n$, then all trigonometric functions in \eqref{eq:d_n} are strictly positive. This implies the lower estimate. To obtain the upper bound, notice that
\begin{equation*}
		d_n (\pi^2/\ell_n^2) = - s_{n+1}\sin \left ( \frac{\ell_{n-1} }{\ell_n}  \pi \right) \leq 0,
	\end{equation*}
	and the equality holds exactly when $\ell_n=\ell_{n-1}$. On the other hand, $d_n(\lambda) >0 $ if $\lambda$ is sufficiently close to zero. Hence $d_n$ has at least one zero in the interval $(0, \pi^2/\ell_n^2]$.
\end{proof}

\begin{remark}
The lower bound in \eqref{est:l1} becomes equality if $\ell_{n-1}= \ell_n$.
 Moreover, upon noting that 
 \[
	\lim_{\ell_{n-1}\to 0 } \lambda_1(\wt{\rh}_n) = \left(  \frac{\pi }{ \ell_n } \right)^2,
\]
for fixed $s_n$, $s_{n+1}, \ell_{n} >0$, the upper estimate in \eqref{est:l1} is also sharp.
\end{remark}

We can slightly improve the upper bound in the following way.

\begin{corollary}
If either $s_{n+1}> s_{n-1}$ and $\ell_n > \ell_{n-1}$ or $s_{n+1}< s_{n-1}$ and $\ell_n < \ell_{n-1}$, then 
\be\label{est:l2}
\lambda_1 ({\wt \rh}_n) < \frac{\pi^2}{(\ell_{n-1} +\ell_n)^2}.
\ee
\end{corollary}

\begin{proof}
Since
\[
\frac{\pi \ell_{n-1}}{\ell_{n-1} +\ell_n} = \pi - \frac{\pi \ell_{n}}{\ell_{n-1} +\ell_n}, 
\]
we immediately get 
\begin{align*}
d_n\Big(\frac{ \pi^2}{(\ell_{n-1} +\ell_n)^2}\Big) 
& = \frac{s_{n+1} - s_{n-1}}{2} \sin\Big(2\pi\frac{\ell_{n}}{\ell_{n-1} +\ell_n} \Big) \\
& = \frac{s_{n-1} - s_{n+1}}{2} \sin\Big(2\pi\frac{\ell_{n-1}}{\ell_{n-1} +\ell_n} \Big).
\end{align*}
Using the same argument as in the proof of the upper bound in \eqref{est:l1}, we arrive at the desired estimates. 
\end{proof}

In general, there is no closed form for zeros of the function $d_n$. It is possible to express it with the help of the arctangent function with two arguments.  Namely, we can simplify $d_n$ a little bit:
\begin{align*}
d_n(z^2) & = s_{n+1}\cos(z\ell_n)\sin(z\ell_{n-1})  + s_{n-1}\cos(z\ell_{n-1})\sin(z\ell_{n}) \\
&= \frac{s_{n+1}-s_{n-1}}{2}\sin(z(\ell_{n-1}-\ell_n)) + \frac{s_{n+1}+s_{n-1}}{2}\sin(z(\ell_{n-1}+\ell_n)). 
\end{align*}

Now  let us apply the following formula
\footnote{See, e.g., \tiny{\url{https://en.wikipedia.org/wiki/List_of_trigonometric_identities\#Linear_combinations}}}:
\be
 a\sin x+b\sin(x+\theta )=c\sin(x+\varphi ),
 \ee
 where
\be
 c=\sqrt {a^{2}+b^{2}+2ab\cos(\theta)}
\ee
and
\be
 \varphi =\operatorname {atan2} \left(b\,\sin \theta ,a+b\cos \theta \right).
 \ee
 Here $\operatorname {atan2} (\cdot,\cdot)$ is the {\em arctangent function with two arguments}\footnote{See, e.g., {\tiny \url{https://en.wikipedia.org/wiki/Atan2}}}. Setting 
 \[
 a= s_{n+1}-s_{n-1}, \quad b= s_{n+1} + s_{n-1},\qquad \theta = 2z\ell_n,
 \]
 we get
 \be
 \sqrt{2}d_n(z^2) = \sqrt{s_{n-1}^2 + s_n^2 +(s_n^2 - s_{n-1}^2)\cos(2z\ell_n)} \sin\big(z(\ell_{n-1}-\ell_n) + \phi(z)\big),
 \ee
 where
 \begin{align*}
 \phi(z) = \operatorname {atan2} \left((s_{n+1} + s_{n-1})\,\sin(2z\ell_n), (s_{n+1} - s_{n-1})+(s_{n+1} + s_{n-1})\cos(2z\ell_n) \right).
 \end{align*}
However, seems this does not lead to a closed formula anyway.

\section{Antitrees with bounded geometry and ac-spectrum}\label{sec:ATbndgeom}

In this section we collect several examples of antitrees such that the degree function is bounded. In fact, in all of the examples the degree function takes finitely many values and hence these examples complement Theorem \ref{th:rem02} in a certain way since the length function is allowed to take infinitely many values.

\begin{example}\label{ex:thac02ex2}
Suppose $p,q \in \Z_{\ge 1}$ and let $\cA$ be the antitree with sphere numbers $s_0=1$ and
\begin{align*}
s_{2n-1} & = q, & s_{2n} & = p,
\end{align*}
for all $n\in\Z_{\ge 1}$. 
Let $\{\ell_n\}_{n\ge 0}$ be edge lengths. If $\cL= \sum_n \ell_n < \infty$, then clearly $\vol(\cA)<\infty$ and the spectra of self-adjoint extensions of $\bH$ are purely discrete by Corollary \ref{cor:finvol}. However, if $\cL = \infty$, then $\mu(x) \equiv pq$ on $I_n$ for all $n\ge 1$ and 
\[
\int_n^{n+2} \mu(x)dx\int_n^{n+2}\frac{dx}{\mu(x)} - 4 = 0
\]
for all $n\ge t_1$. Hence Theorem \ref{th:ac01} applies and thus $\sigma_{\ac}(\bH) = \R_{\ge 0}$.  
\hfill $\lozenge$
\end{example}

\begin{remark}
Since $\mu\equiv const$ on $(t_1,\cL)$ in the above example, the corresponding operator $\rH$ can be considered as a coupling of a weighted Sturm--Liouville operator acting on $(0,t_1)$ and the free Schr\"odinger operator $-d^2/dx^2$ on $(t_1,\cL)$. This explains the result in the above example. Moreover, it also implies that the spectrum of $\rH$ is purely absolutely continuous and coincides with $\R_{\ge 0}$ if $\cL=\infty$. 
\end{remark}

Using Theorem \ref{th:ac01} one can construct non-equilateral antitress satisfying the assumptions of Theorem \ref{th:ac01} and such that $\lim_{n\to \infty} \ell_n = 0$.

\begin{example} \label{ex:thac02ex1}
Let $\{\delta_k\}_{k\ge 0} \in \ell^1$ be a sequence of positive numbers. 
As in Figure \ref{fig:thac02ex1}, let $\cA$ be the radially symmetric antitree with sphere numbers
\begin{align*}
s_{4k} & = s_{4k+1} =1, &  s_{4k+2} & = s_{4k+3} =2	,
\end{align*}
and edge lengths
\begin{align*}
\ell_{4k} & = 1, & \ell_{4k+1} & = \ell_{4k+2} = \ell_{4k+3} =\delta_k,	
\end{align*}
for all $k\ge 0$.
\begin{figure}
	\begin{center}
		\begin{tikzpicture}    [scale = 0.8,every node/.style={scale=.8}]

		
		\foreach \x in {0,3,6}
		{\filldraw 
		(\x, 0) circle (2pt)
		(\x +2 , 0) circle (2pt);
		
		\draw (\x, 0) -- (\x +2, 0);
		}
		
		\filldraw (9, 0) circle (2pt);
		\draw [thin, dashed] (9, 0) -- (10.5, 0);
		
		\foreach \x in {2,5,8}
		\foreach \y in {-1, 1}
		\foreach \z in {1,2}
		{\filldraw (\x + \z* 1/3, \y) circle (2pt);
		}
		
		\foreach \x in {2,5,8}
		\foreach \y in {-1, 1}
		{	\draw (\x, 0) -- (\x + 1/3, \y);
			\draw (\x+1, 0) -- (\x + 2/3, \y);
			\draw (\x+1/3, \y) -- (\x + 2/3, \y);
				\draw (\x+1/3, \y) -- (\x + 2/3, -\y);
		}

		\end{tikzpicture}
	\end{center}
	\caption{The antitree $\cA$ in Example \ref{ex:thac02ex1}} \label{fig:thac02ex1}
\end{figure}
Notice that
\begin{align*}
	\mu(x) = \begin{cases} 1, &   x \in \cup_{k\ge 0}I_{ 4k} \\  2, &   x \in \cup_{k\ge 0}(I_{ 4k + 1}\cup I_{ 4k+3}) \\  4, &  x \in \cup_{k\ge 0} I_{ 4k+2}  \end{cases}.
\end{align*}
Denote $\Omega_1 =  \cup_{k\ge 0}I_{ 4k}$ and $\Omega_2 = \R_{\ge 0}\setminus\Omega_1$ and set 
\begin{align*}
\Omega_1^n & := (n,n+2)\cap \Omega_1, & \Omega_2^n & := (n,n+2)\cap \Omega_2,
\end{align*}
for all $n\ge 0$.
Notice that the sets $\Omega_1^n $ and $\Omega_2^n$ are disjoint and $|\Omega_1^n| + |\Omega_2^n| =2$. Then we can estimate
\begin{align*}
\int_n^{n+2} \mu(x)dx & \int_n^{n+2}\frac{dx}{\mu(x)} - 4   
= \Big(|\Omega_1^n| +  \int_{\Omega_2^n} \mu(x)dx\Big)\Big(|\Omega_1^n| +  \int_{\Omega_2^n} \frac{dx}{\mu(x)}\Big)- 4\\
& \le (2 - |\Omega_2^n| + 4|\Omega_2^n|)(2 - |\Omega_2^n| + \frac{1}{2}|\Omega_2^n|) - 4 \le 5 |\Omega_2^n|.
\end{align*}
Taking into account that 
\[
\sum_{n\ge 0} |\Omega_2^n| = 3\sum_{n\ge 0}\delta_n <\infty,
\]
and then applying Theorem \ref{th:ac01}, we conclude that $\sigma_{\ac}(\bH) = \R_{\ge0}$.
 \hfill $\lozenge$
\end{example}

In fact, Examples \ref{ex:thac02ex2} and \ref{ex:thac02ex1} are in a certain sense very similar and can be generalized to a much wider extent. 

\begin{corollary}\label{cor:ac03} 
	Let $\cA$ be an infinite radially symmetric antitree such that $\{s_n\}_{n\ge 0}$ takes only finitely many different values, that is, the set $\cS := \{ s_n\colon n \in \Z_{\ge 0} \}$ is finite. Suppose further that there exists $\sigma_0 \in \cM$ such that
	\[
	|\cI_{\sigma_0}| =  \infty \qquad \text{ and } \qquad |\cI_{s}| < \infty  \quad \text{ for all } s \in \cM\setminus\{\sigma_0\}.
	\]
	Then $\sigma_{\ac}  (\bH) =\R_{\ge0}$.
\end{corollary}

\begin{proof}
Since $\cM$ is finite by assumption,
	\[
		C := \sup_{s , \xi \in \cM} \frac{ (s - \xi)^2 }{ s \xi}  < \infty.
	\]
Taking into account \eqref{eq:sumIn}, for every $n \in \Z_{\ge 0}$, we get
	\begin{align*}
		\sum_{s \in \cM}  \sum_{\xi \neq s}  |\cI_s^n| |\cI_\xi^n|  &
		=  |\cI_{\sigma_0}^n|  \sum_{\xi \neq \sigma_0}  |\cI_{\xi}^n| 
	+ \sum_{ s \neq \sigma_0 } |\cI_{s}^n| ( 2 -  |\cI_{s}^n|) \\
	&=  \sum_{s \neq \sigma_0}  |\cI_{s}^n|  (  |\cI_{\sigma_0}^n|  + 2 -  |\cI_{s}^n|)  \le 4  \sum_{s \neq \sigma_0}  |\cI_{s}^n|.
	\end{align*}
Therefore, we end up with the following estimate 
	\begin{align*}
	 \sum_{n\ge 0} \sum_{s \in \cM} \sum_{\xi \neq s}  |\cI_s^n| |\cI_\xi^n| \frac{(s - \xi)^2}{s \xi}  
	 \le 4 C  \sum_{n \geq 0}\sum_{s \neq \sigma_0}  |\cI_{s}^n|  \le 8 C \sum_{s \neq \sigma_0} | \cI_s| < \infty.
	\end{align*}
Taking into account Lemma \ref{lem:ac01} and then applying Theorem \ref{th:ac01}, we conclude that $\sigma_{\ac}  (\bH) =\R_{\ge0}$.
\end{proof}

Let us present one more example based on the use of Theorem \ref{th:ac02}.

\begin{example}
Let $\{\beta_k\}_{k \ge 0}$ and $\{\delta_k\}_{k \ge 0}$ be bounded sequences of positive numbers such that
\begin{align*}
	\sum_{k \ge 0} (\beta_k + \delta_k) =  \infty, \qquad \text{ and } \qquad \sum_{k \ge 0} \delta_k^3 < \infty.
\end{align*}
Let also $p,q,r \in \Z_{\ge 1}$ satisfy $p <  r < q$. Consider the following antitree 
\[
s_{3k+1} = q, \qquad s_{3k+2} = r, \qquad s_{3k+3} = p,
\]
for all $k\in\Z_{\ge 0}$.
Next we set
\begin{align*}
	\alpha_0 := q (1-p^2),  \qquad \alpha_1 := \frac{q}{r}(r^2-p^2), \qquad \alpha_2 := \frac{p}{r}(r^2-q^2),
\end{align*}
and equip $\cA$ with edge lengths
\begin{align*}
\ell_{3k} & = \beta_k, & \ell_{3k+1} & =\delta_k, & \ell_{3k+2} & = \frac{\alpha_1}{|\alpha_2|} \delta_k,
\end{align*}
for all $k\in\Z_{\ge 0}$.
Notice that by assumption $\alpha_0\le 0$, $\alpha_1>0$ and $\alpha_2<0$.
Moreover, observe that
\begin{align*}
	 \mu(x) - \frac{(pq)^2}{\mu(x)} = \begin{cases}	0, & x \in \cup_{k\ge1} I_{3k} \\
											\alpha_1, &  x \in \cup_{k\ge0} I_{3k+1} \\
											\alpha_2, &  x \in \cup_{k\ge0} I_{3k+2} 
		\end{cases}.
\end{align*}

Let us now apply Theorem \ref{th:ac02} with 
\begin{align*}
	b & = (pq)^2, & a & = \int_0^{\ell_0} \mu(x) - \frac{(pq)^2}{\mu(x)} \; dx = \ell_0 \alpha_0.
\end{align*}
Due to the choice of edge lengths, we have
\begin{align*}
	\int_0^{t_3}  \mu(x) - \frac{(pq)^2}{\mu(x)} \; dx =  \ell_0 \alpha_0 + \ell_1\alpha_1 + \ell_2 \alpha_2 = a,
\end{align*}
and for every $k \ge 1$, 
\begin{align*}
	\int_{t_{3k}}^{t_{3k+3}} \mu(x) - \frac{(pq)^2}{\mu(x)} \; dx = \ell_{3k+1} \alpha_1 + \ell_{3k+2} \alpha_2 = 0.
\end{align*}
In particular, we obtain for $x\in [t_{3k}, t_{3k+3})$ and $k \ge 1$ 
\begin{align*}
	\int_0^x \mu(x) - \frac{(pq)^2}{\mu(x)} \; dx - a =
	             \begin{cases} 0, & x \in I_{3k} \\
					\alpha_1 (x - t_{3k+1}),	& x \in I_{3k+1} \\
					\alpha_1 \delta_k + \alpha_2 (x - t_{3k+2}),	& x \in I_{3k+2} \end{cases}.
\end{align*}
Thus after some calculations we compute
\begin{align*}
	\int_{t_3}^\cL \frac{1}{\mu(x)}\Big| \int_0^{x} \Big(\mu(s) - \frac{pq}{\mu(s)}\Big)ds - a \Big|^2 dx=\Big ( \frac{\alpha_1^2 }{3 qr} +   \frac{\alpha_1^3}{3 pr |\alpha_2|}  \Big )  \sum_{k=1}^\infty \delta_k^3 <\infty.
\end{align*}
Applying Theorem \ref{th:ac02}, we conclude that $\sigma_{\rm ac} (\bH) = \R_{\ge 0}$. 
\hfill $\lozenge$
\end{example}
\fi


\end{document}